\documentclass[12pt,twoside]{amsart}
\usepackage{amsmath,amssymb,mathrsfs,graphicx,bbm,amsfonts,amsthm, bm}
\usepackage [latin1]{inputenc}
\usepackage{color}

\usepackage{graphicx}
\usepackage{dsfont}

\DeclareMathAlphabet{\mathbfit}{OML}{cmm}{b}{it}

\makeatletter
\@addtoreset{equation}{section}
\makeatother

 \newcommand{\ba}{\begin{array}}
 \newcommand{\ea}{\end{array}}
 \newcommand{\bea}{\begin{eqnarray}}
 \newcommand{\eea}{\end{eqnarray}}
 \newcommand{\be}{\begin{equation}}
  \newcommand{\ee}{\end{equation}}

\usepackage{amsthm,bbm}

\newtheorem{theorem}{Theorem}[section]
\newtheorem{lemma}[theorem]{Lemma}
\newtheorem{proposition}[theorem]{Proposition}
\newtheorem{corollary}[theorem]{Corollary}
\newtheorem{definition}[theorem]{Definition}

\newtheorem{remark}[theorem]{Remark}

 \def \Z {{\mathbb Z}}
 \def \R {{\mathbb R}}
 \def \C {{\mathbb C}}
 \def \N {{\mathbb N}}
 \def \P {{\mathbb P}}
 
 \def \E {{\mathbb E}}

 \def \cG {\mathcal{G}}

 \def \cN {\mathcal{N}}

 \def \cT {\mathcal{T}}

 \def \a {{\alpha}}
 \def \b {{\beta}}

  \def \g {{\gamma}}

 \def \z {{\z}}
 \def \m {{\mu}}
 \def \n {{\nu}}
 
 \def \th {{\theta}}
 
 \def \o {{\omega}}
 
 \def \r {{\rho}}
 \def \z {{\zeta}}

 \def \G {{\Gamma}}

 \def \F {{\Phi}}

\def \à {{\`{a}}}
\def \ì{{\`{\i}}}
\def \ò{{\`{o}}}
\def \è{{\`{e}}}
\def \ù{{\`{u}}}

 \def \1{\mathbbm{1}} 
 
 \def\cov{\hbox{\rm Cov}}
 \def\var{\hbox{\rm Var}}
 \def\sgn{\hbox{\rm sgn}}

\usepackage{wedn}
\usepackage[T1]{fontenc}

\def\sss{{\text{\wedn{s}}}}

\def\ooo{{\text{\wedn{o}}}}

\usepackage{mathtools}
\usepackage{enumitem}
\usepackage{xcolor}

\begin{document}
\title[Ladder costs for random walks in L\'evy random media]{Ladder costs for random walks\\ in L\'evy random media}

\author{Alessandra Bianchi}
\address{Dipartimento di Matematica, Universit\`a di Padova, Via Trieste 63, 35121 Padova, Italy.}
\email{alessandra.bianchi@unipd.it}

\author{Giampaolo Cristadoro}
\address{Dipartimento di Matematica e Applicazioni,
Universit\`a di Milano-Bicocca,
Via R. Cozzi 55, 20125 Milano, Italy.}
\email{giampaolo.cristadoro@unimib.it}

\author{Gaia Pozzoli}
\address{Department of Mathematics, CY Cergy Paris University, CNRS UMR 8088, 2 avenue Adolphe Chauvin, 95302 Cergy-Pontoise, France. }
\email{gaia.pozzoli@cyu.fr}

\date{}

\begin{abstract}
We consider a  random walk $Y$  moving on a \emph{L\'evy random medium}, namely a one-dimensional renewal point process with inter-distances between  points
that are in the domain of attraction of a stable law.
The focus is on the characterization of the law of the first-ladder height $Y_\cT$ and length $L_\cT(Y)$, where $\cT$ is the first-passage
time of $Y$ in $\R^+$.
The study relies on the construction  of a broader class of processes, denoted \emph{Random Walks in Random Scenery on Bonds} (RWRSB) that we briefly describe. The scenery is constructed  by associating two random variables with each bond of $\Z$, corresponding to the two possible crossing directions of that bond.
A random walk $S$ on $\Z$ with i.i.d increments  collects the scenery values of the  bond  it traverses: we denote this composite process the RWRSB.
Under suitable assumptions,
we characterize the tail distribution of  the sum of scenery values collected up to the first exit time $\cT$. 
This setting will be applied to obtain  results for the laws of the first-ladder length and height of $Y$.
The main tools of investigation are a generalized Spitzer-Baxter identity,
that we derive along the proof, and a suitable representation of the RWRSB
in terms of local times of the random walk $S$. All these results are easily generalized to the entire sequence of ladder variables.

    \end{abstract}

\maketitle

\section{Introduction}
An essential component of fluctuation theory in discrete time
is the study of the first-ladder height and time of a one-dimensional (1D)
random walk $S= (S_n)_{n\in\N_0}$, respectively given by the first maximal value
reached by $S$, and by the corresponding time.
In this context, the Wiener-Hopf techniques, introduced by Spitzer, Baxter and others,
offer a main tool of investigation, and allow for the derivation
of several fundamental identities that relate the distributions
of these first-ladder random variables  to that of the underlying random walk
(see \cite{Sp56, Sp57, Sp60}, and \cite{RWFT, Chung2001} for reviews).
These results, which have been established in different formulations by many authors,
open the way to a refined understanding of the first-ladder quantities of the ladder (ascending or descending) process, and of conditioned random walks (\cite{D95, DG93, AD, GOT, Cara}).

The aim of the present work is to generalize these kind of results
to random walks moving on a one-dimensional random medium having
i.i.d.\ inter-distances.
Specifically, the random medium that we consider is a renewal point process
$\omega = ( \omega_k)_{k \in \Z}$  with i.i.d.\ (positive) inter-distances
in the normal domain of attraction of a stable variable. The random walk  on the random medium
$\omega$  is then defined as $Y = (Y_n)_{n \in \N_0}$ , where $Y_n := \omega_{S_n}$ and  $S = ( S_n )_{ n \in \N_0}$
is  an underlying  random walk on $\Z$, independent of $\omega$.

Whenever the inter-distances are heavy-tailed random variables, 
the process $Y$ can be seen as a generalization of 
\emph{L\'evy flights}, which are random walks with  i.i.d.\ heavy-tailed jumps,
and as a discrete time version of the
\emph{L\'evy-Lorentz gas} (see \cite{BFK} and \cite{bcll, blp} for some related extensions).
All these processes have been receiving a surge of attention as
they model phenomena of anomalous transport and anomalous diffusion
(see, e.g. \cite{KRS, ZDK, acor, roac, Z} for some general or recent references).

From its definition, it turns out that the process $Y$ performs the same jumps
as $S$ but on the marked points of $\omega$ instead of $\Z$. Thus the first-ladder times of $S$ and $Y$ correspond, and we set
$$\cT:=\min\{n>0\,:\, S_n>0\}\equiv \min\{n>0\,:\, Y_n>0\}\,.$$
A complete characterization of the law of $\cT$ can then be derived
from the classical Spitzer-Baxter identities, and specifically from
the so called Sparre-Andersen identity \cite{SA1,SA2}
(see also \cite{Chung2001} for a general treatment, and \cite[\S\,17]{Sp76Book} for the specific lattice case).
In particular, if the random walk $S$ has symmetric jumps as in our definition,
then the law of $\cT$   is in the domain of attraction of a $1/2$-stable law.

The characterization of the distribution of $Y_{\cT}$,
the first-ladder height of $Y$, is in general an open problem,
as the double source of randomness creates a non-trivial dependence
between the increments of the process and makes the analysis of the corresponding
motion much harder than the classical independent case,
studied for example in \cite{S57, DG93}.
For the same reason, the law of the ascending ladder process
$(Y_{\cT_k})_{k\in\N_0}$, where $\cT_k$ is the time corresponding to
the $k$-th maximum value reached by $Y$, that is
$$\cT_0=0\,,\quad \cT_k:=\min\{n> \cT_ {k-1}\,:\, Y_n > Y_{\cT_{k-1}}\}
\quad \forall k\in\N\,$$
(with $\cT_1\equiv \cT$),
is in general unknown.

We will approach the problem by considering a slightly generalized
setting, appearing in several applications, in which a cost process $C$
is associated with a real (continuous or discrete) random walk $S$, that is assumed to be
the control process.
As a simple but paradigmatic example, suppose that each jump of the random walk takes a given and possibly random cost (e.g. time or energy) to be performed.
We could then be interested in the total cost accumulated when the walk reaches its first maximum, that is, the quantity $C_{\cT}$.
As a first step of our analysis, we will derive
a generalized  Spitzer-Baxter identity for $(\cT,C_{\cT})$ under the
assumption that the process $C$ has i.i.d.\ increments, possibly depending on
the control process (Theorem \ref{GSB}).
When the cost process is chosen to be exactly equal to $S$, we  recover the classical
Spitzer-Baxter identity. With different choices of $C$, we can
derive information on different types of first-ladder random variables
associated with the process, such as its first-ladder length.

We then move to \emph{Random Walks in Random Scenery on Bonds}. In this setting, as already mentioned, at each step the walker collects the scenery values
of the bond it traverses. It is manifest that these quantities can be seen, in the same spirit as in the previous part, as increments of a cost process $C$ associated with $S$.
On the other hand, such increments are now not i.i.d.\ and thus Theorem \ref{GSB} cannot be applied directly.
 However, assuming that $S$ has symmetric increments
and thanks to a representation of the process in terms of the local times of $S$,
we will be able to express the generating function of $C_\cT$
in a simpler form, which allows the implementation of the generalized
Spitzer-Baxter identities. The explicit results are derived from Tauberian theorems
under mild assumptions about the scenery process.
Then we will adapt the techniques used to analyze the first-ladder quantities
to obtain analogous results for the $k$-th ladder costs, $C_{\cT_k}$.

Finally, these results are applied to derive the tail distributions of the first-ladder length and height of the process $Y$.
In turn, the latter can be used to infer the law of the first-passage time of a generalized L\'evy-Lorentz gas.
\\

The paper is organized as follows.
Section \ref{MeB} is devoted to the rigorous definition of cost and control processes, random walks in random scenery on bonds and the related first-ladder quantities. At the same time, we provide the statement of the associated main results.
All the proofs of these theorems are presented in Section 3, together with some explicit applications to random walks in L\'evy random media and L\'evy-Lorentz gas.

\section{Setup and Main Results}\label{MeB}
Let us consider a process
$S:=( S_n)_{n \in \N_0}$ taking values on $\R$ and, for a fixed $\ell\in\N$,
a $\ell$-dimensional process  $C:=(C_n)_{n \in \N_0}$, which could depend non-trivially on $S$. We denote by $(\xi_k,\eta_k)_{k\in\mathbbm{N}}$ 
 the increments of the joint process $(S,C)$.
$C$ is referred to as the \emph{cost process} while $S$ is the \emph{control process}\footnote{It is apparent that we can equivalently define an  $(\ell+1)$-dimensional process and choose an arbitrary coordinate to play the role of the control process and the remaining $\ell$'s  as the cost process. We prefer to stick to a more explicit notation for the sake of  clarity.}.  
We avoid explicitly giving the dependencies on $S$ of the cost process, unless necessary.
 \\

\noindent \textbf{Example 1.} \emph{As a simple but paradigmatic example  to be used in next sections, consider the  one-dimensional cost process obtained by choosing   $\eta_k=|\xi_k|$, that is
\be\label{costLength}
C_n(S)=\sum_{k=1}^{n} |S_k-S_{k-1}|\,=:L_n(S),\qquad \forall n\in \N_0 \,.
\ee
It is manifest that  $L_n(S)$ measures the total length of the walker after $n$ steps.
}\\

We define  the \emph{first-ladder time} in  $(0,\infty)$ of $S$
(or first-passage time  of $S$) as
 \be\label{def-FP}
 \cT := \min\{ n>0 \: : \: S_n > 0 \}\,
 \ee
and the corresponding \emph{first-ladder height (or leapover)} as the control process stopped at $\cT$, i.e. $S_\cT$. In the same spirit,  we can define the
\emph{first-ladder cost} as the value of the cost process ${C}$  stopped at $\cT $, i.e. ${C}_\cT$.
With the choice \eqref{costLength},
$L_{\cT}(S)$ is the  \emph{first-ladder length} of  $S$, that is, the length of the process $S$ up to its first-passage in $(0,\infty)$.
 \\

We now give an overview of our main results, characterizing the law of $C_\cT$ under different assumptions.  The first result is an explicit expression
for the joint generating function of $(\cT, S_\cT, {C}_\cT)$
 under the assumption that the joint process $(S,{C})$ has i.i.d.\ increments.
This result will be instrumental for the analysis of first-ladder
quantities related to the random walk in random media $Y$,
though not directly applicable to it.
We then introduce a general process, called \emph{Random Walk in Random Scenery on Bonds}  from its analogy to Random Walks in Random Scenery \cite{KS}, so to obtain a suitable representation of $Y_\cT$ and of $L_\cT(Y)$ in this setting. This new process, for which we will state our main result, can be seen as a cost process coupled with $S$, having dependent increments also depending on a
random scenery assigned to the bonds of $\Z$.
Finally, as explicit applications of the main result, we derive the law of
the first-ladder quantities for the random walks in L\'evy random media.
\newline

\subsection{Cost process with i.i.d.\ increments}
The investigation of first-ladder time  and  height  of a 1D random walk is nowadays a well-established topic in fluctuation theory.  Among well-known results,  that are derived under the assumption of independent and identically distributed increments of the walk $S$, the Spitzer-Baxter identity  provides an explicit formula for the generating function of the  first-ladder time $\cT$ and height $S_{\cT}$:
\begin{eqnarray}\label{classicalSB}
\mathbb{E}\left[z^{\cT} e^{it S_{\cT}}\right]
&=& 1-\exp \left(-\sum_{n=1}^{\infty} \frac{z^{n}}{n} \int_{\left\{S_{n} > 0\right\}}
e^{i t S_{n}} d \mathbb{P}\right)\,.
\end{eqnarray}

In the same spirit of these classical results,  assuming that the process $(S, {C})$  is the sum of i.i.d.\ random variables\footnote{It is worth stressing that we are not assuming that ${\eta}_i$ and $\xi_i$ are independent.}
$(\xi_k, {\eta}_k)_{k \in \N}$
we derive an identity akin to the Spitzer-Baxter identity for the joint control and cost processes:

\begin{theorem}\label{GSB}
Suppose that the joint  process  $(S, {C})$ has i.i.d.\ increments.
Then, for any $t\in \R$,  $s \in \R^{\ell}$  and $z\in(0,1)$,
\begin{eqnarray}
\label{gSB-1}
\mathbb{E}\left[ z^{\cT} e^{i t S_{\cT}} e^{i {s}\cdot {C}_{\cT}} \right]
&=& 1-\exp \left(-\sum_{n=1}^{\infty} \frac{z^{n}}{n} \int_{\left\{S_{n} > 0\right\}}
e^{i t S_{n}}e^{i {s} \cdot {C}_{n}} d \mathbb{P}\right)
\,,
\\
\label{gSB-2}
\mathbb{E}\left[\sum_{n=0}^{\cT-1} z^{n}  e^{i t S_{n}} e^{i {s}\cdot {C}_{n}}  \right] &=&\exp \left( +\sum_{n=1}^{\infty}  \frac{z^{n}}{n} \int_{ \left\{ S_{n} \le 0 \right\}} e^{i t S_{n}}e^{i {s} \cdot {C}_{n}} d \mathbb{P}\right)\,.
\end{eqnarray}	
 \end{theorem}
\vspace{1cm}
As in the classical setting, the identities (\ref{gSB-1}) and (\ref{gSB-2}) allow to determine the laws of $\cT, S_\cT, C_\cT$ from the knowledge of quantities that do not depend on the first-passage event (right-hand sides of the identities).

\subsection{Random Walk in Random Scenery on Bonds}
Let $\zeta^{\pm} = (\zeta^{\pm}_k)_{k \in \Z} $ be two sequences of i.i.d.\ real-valued random variables defining the random scenery:  $\zeta^{+}_k$
and $\zeta^{-}_k$ are the values of the scenery at bond $k$ --- the edge between $k-1$ and $k$ --- of $\Z$.
In the following, we consider the random walk 
$S$ with i.i.d.\ symmetric increments $\xi_k$'s
taking values in \( \mathbb{Z} \), 
independent of the sequences \( \zeta^{\pm} \). 
The sequences \( \zeta^{\pm} \) may, in general, be dependent, as illustrated in the applications discussed in Subsection \ref{subsec_firstladder}.

We then consider the cost process $C= (C_n)_{n\in\N_0}$, depending
on $S$ and  $\zeta^{\pm}$,  such that $C_0=0$ and, for $n\in\N$,
\be
\label{cost-RWRE}
C_n : =\sum_{k=1}^n \eta_k\,,\quad
\qquad\mbox{ with } \quad
 \eta_k=\begin{cases}
 \displaystyle\sum_{j=S_{k-1}}^{S_k-1}\z^+_{j+1}\,,\qquad &  \textrm{if}\quad \xi_k>0\,,\\
 \\
  \displaystyle 0 &  \textrm{if}\quad \xi_k=0\,,\\
  \\
 \displaystyle\sum_{j=S_k}^{S_{k-1}-1}\z^-_{j+1}\,,\qquad &  \textrm{if}\quad \xi_k<0\,.
\end{cases}
\ee
Basically, each $\eta_k$  collects all the scenery values corresponding to the bonds that have been crossed in the corresponding jump $\xi_k$ of $S$, while $\zeta^{+}$ determines the weight associated with the bond traversed to the right and $\z^-$ with the bond traversed to the left.
In particular, the cost process $C$ depends on both $\z^{\pm}$ and $S$ and is called Random Walk in Random Scenery on Bonds (RWRSB). 
For specific choices of the random scenery, this class of processes includes  
the family of random walks on L\'evy media on which Subsection 2.3 is focused.

Note that the presence of the random scenery breaks down the i.i.d.\ assumption of the
 generalized Spitzer-Baxter identity stated in Theorem \ref{GSB}.
We will show how to leverage the results for first-ladder quantities associated with the control process $S$ to infer properties on the first-ladder quantity $C_\cT$ in this setting. 
As will be clear in the proof, to characterize  $C_\cT$  we need to also consider the even part of the scenery random variables, defined as 
$$\zeta^0_k\coloneqq \frac{\zeta^+_k+\zeta^-_k}2, \quad \forall k \in \Z\,.$$
We are then led to work with the following general assumptions.
\\
\textbf{Assumptions.}
\begin{itemize}
\item[a1.] Assume that S is a random walk on $\Z$ with  i.i.d.\ symmetric increments in the normal 
basin of attraction of a  $\beta$-stable distribution, with $\beta\in(0,2)$.
\item[a2.] Assume that the random variables $\z^+_k$'s and $\z^0_k$'s are non-negative (or non-positive),
and that they belong to the normal domain of attraction of stable random variables (including the degenerate case) with indices 
$\gamma_+\in(0,2]$ and $\gamma_0\in(0,2]$ respectively.
\end{itemize}
We refer the reader to Appendix \ref{app-stable}, where we have gathered essential definitions 
and results on random variables in the domain of attraction of a stable distribution, 
which will be used throughout the paper.

To state the main theorem, it is also convenient to define the following constants
that will be used throughout the paper:\\
\textbf{Notation.}
\begin{itemize}
\item If  $X$  is a random variable in the domain of attraction of an $\alpha$-stable law, 
with $\alpha \in (0,2]$, we set $\hat{\alpha}:=\textrm{min}\{1,\alpha\}$. 
If $X\equiv0$, we set $\hat{\alpha}=+\infty$;
\item In the above setting, we write
\begin{equation}\label{notation}
\r_+\coloneqq \hat{\g}_+\,\b /2,   \quad \, \r_0 \coloneqq \hat{\g}_0\, \hat{\b}/2.
\end{equation}
\end{itemize}
Note that $\r_+$ as well as $\r_0$ involve the stability indexes of both the scenery values and the underlying random walk. 
Indeed, we heuristically expect that the asymptotic tail of the ladder cost should receive contributions from both the elements of randomness, as the RWRSB interlaces 
(otherwise independent) processes.  
To better grasp the role of such exponents,  it is helpful to gain some intuition about the different terms that contribute to the cost $C_\cT$ accumulated up to the first-ladder time. 
First of all note that, for the part of the walk in the negative semi-axis, the random walker  traverses each bond an even number of times. Indeed, all bonds in the negative semi-axis that are traversed once to the left must be also traversed once to the right since the walker has to eventually pass over the origin to first land on the positive semi-axis. In this part of the walk, the effective cost of the $k$-th bond is thus the \emph{average} $\zeta^0_k$ of the directional costs, which has stability index ${\gamma}_0$.
These averaged costs will be collected for a number of times equal to $(L_\cT-S_\cT)/2$, which has a power-law decay with exponent $\hat{\beta}/2$ (see Lemma~\ref{lemma 1}). Hence, the contribution to $C_\cT$ from the part of the walk on the negative semi-axis combines these two terms, resulting (for technical reasons) in a stability index $\rho_0$.
In contrast, for the segment of the walk on the positive semi-axis, there are $S_\cT$  bonds that are traversed only once and in the right direction. This means that for this part only $\zeta^+$ counts, with stability index ${\gamma}_+$, and since $S_\cT$ has stability index $\beta/2$, this contribution to $C_\cT$ will have stability index $\rho_+$.
The technical reasons behind the introduction of the indices $\hat\gamma_0,\hat\gamma_+$ can be motivated by the following observation: 
when $\E(\z^{\pm}_1)<\infty$, as expected, the exponent of the asymptotic tail of the first-ladder cost is ruled solely by the properties of the underlying random walk $S$.
From this heuristics, we expect that the asymptotic tail of the first-ladder cost is the result of the competition between the two contributions $\rho_0$ and $\rho_+$ described above, 
as substantiated in the proof of the following theorem.\footnote{Throughout this paper, given two functions $f(x)$ and $g(x)$ we write $f(x)\sim g(x)$ if $\lim_{x\to\infty}f(x)/g(x)=1$. Moreover, we write $f(x)\asymp g(x)$ if $\exists\,c_1,c_2>0$ such that $c_1\leq \liminf_{x\to\infty}f(x)/g(x)\leq \limsup_{x\to\infty}f(x)/g(x)\leq c_2$.}
\begin{theorem}\label{MainResult1} 
 Let $C$  be the RWRSB defined in Eq.~\eqref{cost-RWRE} under Assumptions a1. and a2.
Then, there exist slowly varying functions $K(x),K_1(x)$ and  $K_2(x)$ such that
\begin{itemize}
\item  if  $\r_+ <  \r_0$, we have
\[
\P(C_{\cT}>x)\sim K(x)\,x^{-\r_+}\,, \qquad \mbox{as } x\to\infty\,;
\]
\item if  $\r_+\ge \r_0$ and $\beta\in[1,2)$, we get
\[
K_1(x)\,x^{-\min\{\r_+,\hat\g_0,1/2\}}\leq \P(C_{\cT}>x)\leq K_2(x)\,x^{-\r_0}\,,\qquad \mbox{as } x\to\infty\,,
\]
\end{itemize}
where for all  $\g_0\in[1,2]$, the lower tail exponent  matches 
$\r_0=1/2$ .
\end{theorem}
\vspace{0.5cm}

We anticipate that a slightly more general theorem is valid under suitable assumptions (see Theorem~\ref{MainResultExtended} 
and Remark \ref{rmk:cauchy}):
for the ease of the reader, we postpone the technical details and the complete statement to the dedicated section.

\subsection{First-ladder quantities for the random walks in L\'evy random media $Y$.
}\label{subsec_firstladder}
Let $\z := (\z_k)_{k \in \Z}$ be a sequence of i.i.d.\ positive random variables,
whose common distribution belongs to the normal basin of
attraction of a $\gamma$-stable distribution, with $0 < \g \leq 2$ and
 $\g\neq 1$ for simplicity (see Remark~\ref{rmk:cauchy}).
The recursive sequence of definitions
\be \label{PP}
  \o_0 := 0 \, , \qquad \o_k - \o_{k-1} := \z_k \,, \quad \mbox{for } k
  \in \Z \,,
\ee
determines a marked \emph{point process} $\o := ( \o_k)_{k\in\Z}$
on $\R$, which we call the \emph{random medium}.
For a fixed $\o$, and a random walk $(S_n)_{n\in\N_{0}}$ on $\Z$ as given in Assumption a1. above, we define the discrete time process
$Y := (Y_n)_{n \in \N_0}$ setting
\be\label{Y}
  Y_n \equiv Y_n(\o,S) := \o_{S_n}\qquad \forall n \in \N_0\,.
\ee
In simple terms, $Y $ performs the same jumps
as $S$ but on the points of $\o$, thus it is called \emph{random walk on the random medium}.

The presence of the random medium creates a dependence between the
 increments  of the process and provides a more realistic model of motion in
  inhomogeneous media with respect to the classical hypothesis of i.i.d.\ jumps.
However, as before, the double source of randomness makes the analysis of the model
much harder than the classical independent case, and even standard results of the classical theory of random walks, such as central limit theorems, have been only recently obtained under suitable
hypotheses (\cite{bcll,br}).
On the other hand, it is easy to see that the first-ladder time  $\cT$ is the same for both $Y$ and the underlying random walk $S$.
Our aim is thus to characterize the asymptotic law of the first-ladder height
and length of $Y$ (see Example 1.), denoted respectively by $Y_\cT$ and $L_\cT(Y)$.

Note that $(Y_n)_{n\in\N_0}$ can be seen as a RWRSB driven by $S$ with scenery 
$\z^+=-\z^- \equiv \z$, and hence with indices $\gamma_+=\gamma$ and $\hat\gamma_0=+\infty$, which ensures
that $\r_+<\r_0$ (see notation \eqref{notation}).\\
Similarly, $(L_n(Y))_{n\in\N_0}$ can be seen as a RWRSB driven by $S$ with scenery $\z^+=\z^-\equiv\z$, 
and therefore with indices $\g_+=\g_0=\g$, which imply $\r_+\ge \r_0$ (see notation \eqref{notation}).  
As an application of Theorem \ref{MainResult1}, we then get the following results:
\begin{corollary}\label{leapoverY-app}
In the above notation, for any $\beta\in(0,2)$ and $\gamma\in(0,2]\backslash\{1\}$, 
it holds that
\be
\P(Y_\cT> x) \sim Kx^{-\hat{\g}\b/2}\,,
\qquad \mbox{as } x\to\infty\,\,.
\ee
where $K$ is an explicit constant (see Eqs. \eqref{leapoverY-app-case1}, \eqref{leapoverY-app-case2}).
\end{corollary}

\begin{corollary}\label{lengthY}
In the above notation, for any $\beta\in[1,2)$ and $\gamma\in(0,2]\backslash\{1\}$, 
it holds that
\begin{equation}
K_{low}(x)\ \,x^{-\min\left\{1/2,\hat{\g}{\b}/2 \right\}} \leq \P[L_{\cT}(Y)>x]\leq K_{up}(x)\,x^{-\hat{\g}/2},
\end{equation}
where $K_{low}(x)$ and $K_{up}(x)$ are positive constants if $\b\neq 1$, and 
suitable slowly varying functions  if $\b=1$ (see Subsection \ref{GP}).
\end{corollary}

\noindent
Notice that, except for case $\g \in (0,1) $, the decay exponents
for the lower and upper bounds match and equal $1/2$.
\\

We are also interested in the continuous-time process
$X := ( X_t)_{ t \ge 0 }$, whose trajectories interpolate
those of the walk $Y$ and have unit speed.
Formally it can be defined as follows: given a realization $\o$ of the
medium and a realization $S$ of the dynamics, we define the sequence
of \emph{collision times} $T_n \equiv T_n(\o, S)$ via
\be \label{collisiontime}
  T_0 := 0 \,, \qquad T_n := \sum_{k=1}^n
  |\o_{S_k}-\o_{S_{k-1}}| \,, \quad \mbox{for } n \ge 1 \,.
\ee
Since the length of the $n^\mathrm{th}$ jump of the walk is
given by $|\o_{S_n}-\o_{S_{n-1}}|$, $T_n$ represents the global length
of the trajectory Y up to the $n^\mathrm{th}$ collision. In other words, $T_n=L_n(Y)$, and it can be seen as a RWRSB (see also \cite{blp}).
Finally, $X_t \equiv X_t(\o,S)$ is defined by the equations
\be\label{process}
  X_t := Y_n+ \mathrm{sgn}(\xi_{n+1} )(t - T_n) \,, \quad \mbox{for } t \in
  [T_n , T_{n+1}) \,.
\ee
The process $X$ is also important  from the standpoint of applications
as it is a generalization of the so-called
L\'{e}vy-Lorentz gas \cite{BFK}, that is obtained under the further assumption
that the underlying random walk is simple and symmetric.

Functional limit theorems for the processes $Y$ and $X$, with suitable scaling,
 have been derived in \cite{bcll, blp, bblms} under different set of hypotheses.
In particular, when $\g\in (0,1)$ or when the underlying random walk performs
heavy-tailed jumps, the processes $Y$ and $X$ are shown to exhibit an interesting super-diffusive behavior \cite{blp, bblms}.

Let us define  the \emph{first-passage time} in  $(0,\infty)$  by
\be
 \cT(X) := \inf\{ t>0 \,: \, X_t > 0 \}\,.
 \ee
Notice that in this continuous setting the notion of first-ladder height becomes trivial,  while that  of first-ladder length of $X$ indeed corresponds to $\cT(X)$,
being the speed of the process $X$ set equal to 1.
By construction,  and using the previous notation, it can be seen that
\be\label{FPtimeX}
\begin{split}
&\cT(X)= \sum_{k=1}^{\cT} |Y_k-Y_{k-1}| -Y_\cT = L_{\cT}(Y)- Y_{\cT}\, .
\end{split}\ee
This relation shows that, beyond their intrinsic interest, the derivation of the law of
the first-ladder height and length of $Y$ will  allow to infer information on the first-passage time of the process $X$.
Indeed, the continuous first-passage time $L_\cT(Y)-Y_\cT$ can be seen as
the value  at time $\cT$ of the RWRSB driven by $S$ and with scenery 
$\z^+\equiv 0\,$ and  $\z^-= 2\z$, and hence with indices $\g_0=\g$ and $\hat\g_+=+\infty$, 
which imply $\r+> \r_0$ (see notation \eqref{notation}).
 As an application of Theorem \ref{MainResult1} we have the following:
\begin{corollary} \label{gLG}
In the above notation, for any $\beta\in[1,2)$ and $\gamma\in(0,2]\backslash\{1\}$, 
it holds that
\begin{eqnarray}
K_{low}(t)\ \,t^{-\min\left\{\hat{\g},1/2 \right\}}
\leq \P({\cT}(X)>t)\leq K_{up}(t)\,t^{-\hat{\g}/2},
\end{eqnarray}
where $K_{low}(t)$ and $K_{up}(t)$ are positive constants if $\g\in(1,2]$ and $\b\neq 1$, and are suitable slowly varying functions  if $\g\in(0,1)$ or $\b=1$ (see Subsection \ref{sec: appLLG}).
\end{corollary}


\section{Proofs of results}
We give the proof of the main results described in the previous section, together with some useful corollaries. We also discuss  some applications.

 \subsection{Results in the case of $(S, C)$ with i.i.d.\ increments}
In this section we assume that the process $(S,C)$ has i.i.d.\ increments and we introduce the characteristic functions
\be\label{genFunctIncrement}
\phi_{\xi_1,\eta_1}(t,{s}):=\mathbb{E}\left[ e^{i(t\xi_1+{s} \cdot {\eta_1})} \right]\,,
\qquad \phi_{{\eta_1}}({s}):=\phi_{\xi_1,{\eta_1}}(0,{s})\qquad
\mbox{with }t\in\R,\, s\in\R^{\ell}\,.
\ee

We start proving  the generalized Spitzer-Baxter identities stated in Thm.~\ref{GSB}.  The proof follows the line of that for the classical Spitzer-Baxter identity as in \cite[Paragraph 8.4]{Chung2001}. \\
\\
\emph{Proof of Thm. \ref{GSB}:}{
 As $( \xi_k, {\eta}_k)_{k\in\N}$ are i.i.d.\ random variables, we have
\begin{equation}\label{firstFactorization}
\mathbb{E}\left[ \sum_{n=0}^{\infty} z^n  e^{i t S_n} e^{i {s} \cdot {C}_n} \right]= \frac{1}{1-z \phi_{\xi_1,{\eta_1}}(t,{s})}= f^{-1}_{+}(z, t, s) f_{-}(z, t, s)
\end{equation}
where
\begin{subequations}\label{expFunc}
\begin{align}
f_{+}(z,t,s)&:=\exp \left(-\sum_{n=1}^{\infty} \frac{z^{n}}{n} \int_{\{S_n >0\}} e^{itS_n}e^{i {s} \cdot {C}_n }d \mathbb{P} \right),\\
f_{-}(z,t, s)&:=\exp \left(+\sum_{n=1}^{\infty} \frac{z^{n}}{n} \int_{\{S_n  \le 0 \}} e^{itS_n}e^{i {s} \cdot {C}_n }d \mathbb{P}\right).
\end{align}
\end{subequations}
Split the sum in the left-hand side of \eqref{firstFactorization} as
\begin{eqnarray}
\mathbb{E}\left[ \sum_{n=0}^{\infty} z^n  e^{itS_n} e^{i {s} \cdot {C}_n} \right] &=& \mathbb{E}\left[ \sum_{n=0}^{\cT-1} z^n  e^{itS_n} e^{i {s} \cdot {C}_n} \right]  +\mathbb{E}\left[ \sum_{n=\cT }^{\infty} z^n e^{itS_n} e^{i {s} \cdot {C}_n}\right ]\,. \label{splitSum}
\end{eqnarray}
The second term on the right-hand side of \eqref{splitSum} can be rewritten as
\begin{eqnarray}
\mathbb{E}\left[\sum_{n=\cT }^{\infty} z^n  e^{itS_n} e^{i {s} \cdot {C}_n}\right ]&=&\mathbb{E}\left[z^{\cT} e^{itS_{\cT}} e^{i {s} \cdot  {C}_{\cT} } \sum_{n=0 }^{\infty} z^n  e^{it( S_{n+\cT}-S_{\cT})} e^{i {s} \cdot ( {C}_{n+\cT}-{C}_{\cT})} \right]\nonumber\\
&= &\mathbb{E}\left[z^{\cT}  e^{itS_{\cT}}  e^{i {s} \cdot  {C}_{\cT} }\right ]/ \left(1- z \phi_{\xi_1,{\eta_1}}(t,{s}) \right)\,, \label{secondFact}
\end{eqnarray}
where in the last passage we use the fact that $({S}_{n+\cT}-{S}_{\cT},C_{n+\cT}-C_\cT)_{n\in\N_0}$ is independent of $({S}_{\cT},C_\cT)$ and distributed as $({S}_n,C_n)_{n\in\N_0}$.
By using \eqref{firstFactorization} and \eqref{secondFact} in \eqref{splitSum} we get
\begin{eqnarray}
\left[ 1-  \mathbb{E}\left[z^{\cT}  e^{itS_{\cT}}  e^{i {s} \cdot  {C}_{\cT} }\right]\right] f_{-}(z,t, s)  = f_{+}(z, t, s)  \, \mathbb{E}\left[ \sum_{n=0}^{\cT-1} z^n  e^{itS_n} e^{i {s} \cdot {C}_n} \right]\,. \label{WHequality}
\end{eqnarray}
We now apply  standard Wiener-Hopf argument:
  the  convolution of  two measures restricted to $(0,+\infty)$  remains restricted to $(0,+\infty)$  (and the same for $(-\infty,0]$); by expanding the exponential functions in \eqref{expFunc}, we can associate $f_+$ and $f_-$ with $P^*$ and $Q^*$ in Lemma \ref{Wiener-Hopf} respectively, and similarly the remaining terms on both sides of \eqref{WHequality} correspond to $P$ and $Q$. The results \eqref{gSB-1} and \eqref{gSB-2} immediately follow using the lemma.
\qed
\vspace{0.5cm}

Theorem \ref{GSB} is particularly useful when the right-hand side of the identities \eqref{gSB-1} and \eqref{gSB-2} can be computed explicitly.
This happens, for example, when the law of the joint process $(S,C)$ satisfies
some symmetry property. In particular the following definitions will be helpful.

\begin{definition}\label{symm-asymm}
The joint process $(S,C)$ is called  \emph{{\sss}-symmetric} if, for all $   x\in\R$, ${y}\in\R^{\ell}$ and $n\in\N$,
\be\label{simm}
 \, \mathbb{P}(S_n \in dx, {C}_n \in d{y}) = \mathbb{P}( -S_n\in dx, {C}_n \in d{y})\,.
\ee
The joint process $(S,C)$ is called  \emph{{\ooo}-symmetric} if, for all  $ x\in\R$, ${y}\in\R^{\ell}$ and $n\in\N$,
\be\label{asimm}
\mathbb{P}(S_n \in dx, {C}_n \in d{y})
= \mathbb{P}( -S_n\in dx, -{C}_n \in d{y})\,.
\ee
\end{definition}
To simplify the notation we also define, for any  $s \in \R^{\ell}$  and $z\in(0,1)$, the function
\be\label{phi}
\F(z,{s}):=  \exp \left(\frac 12\sum_{n=1}^{\infty} \frac{z^{n}}{n} \int_{\left\{ S_{n} = 0\right\}} e^{i {s} \cdot {({C}_{n}^{\sss} +{C}_n^{\ooo})}} d \mathbb{P}\right)\,.
\ee
\begin{remark}\label{rem-phi}
Note that if the $\xi_k$'s have an absolutely continuous distribution, then one trivially gets
that $\F(z,{s})=1$ for any  $s \in \R^{\ell}$  and $z\in(0,1)$.
On the other hand, if the $\xi_k$'s are discrete random variables taking value on $a\Z$, for $a>0$,
the integral in the right-hand side of \eqref{phi}  is the anti-transform  with respect to $S_n$,  evaluated at $S_n=0$, 
of the joint transform of a $n^\mathrm{th}$ convolution of $(\xi, {\eta})$.
By  the i.i.d.\ assumption on the increments  $(\xi_k, {\eta}_k )_{k\in \N}$,
and the fact that the convolution becomes a product in the transform domain, 
we can then derive the following convenient identity 
(see also~\cite[\S\,17.E1, Eq.\,(6)]{Sp76Book})
\begin{eqnarray}
\F(z,s) &=&
\exp \left( \frac 12 \sum_{n=1}^{\infty} \frac{z^n}{n}\frac{a}{2\pi}\int_{-\frac{\pi}{a}}^{\frac{\pi}{a}} 
\phi_{\xi_1,\eta_1}(t,{s})^n dt \right) \nonumber\\
&=&\exp\left( -\frac{a}{4\pi} \int_{-\frac{\pi}{a}}^{\frac{\pi}{a}} \ln[ 1- z\phi_{\xi_1,\eta_1}(t,{s})]dt \right)\,.
\end{eqnarray}
\end{remark}

\medskip
It is easy to see that  the generalized Spitzer-Baxter can be directly used  to give the following explicit relation involving both  $\sss$-symmetric  and
$\ooo$-symmetric processes, with i.i.d.\ increments.

\begin{corollary}\label{cor-mix}
Consider a joint process $(S,C^\sss + C^\ooo)$, with i.i.d.\ increments, satisfying the symmetry property $(S_n,C_n^\sss,C_n^\ooo)\stackrel{\text{d}}{=}(-S_n,C_n^\sss,-C_n^\ooo)$ for all $n\in\N$.
Then, for all $z\in(0,1)$ and $s\in\R^{\ell}$, 

\begin{equation}\label{mix_cor}
\left(1-\mathbb{E}\left[z^{\cT} e^{i {s}\cdot ({C}_{\cT}^{\sss} +{C}_\cT^{\ooo})}\right] \right)  \left(1-\mathbb{E}\left[z^{\cT} e^{ i {s}\cdot  ({C}_{\cT}^{\sss} -{C}_\cT^{\ooo})}\right] \right) =\left(1-z \phi_{{\eta_1}}({s})\right)\F^2(z,{s})\,,
\end{equation}
where  $\eta=(\eta_k)_{k\in\N}$ are the cost increments associated with ${C}^{\sss} +{C}^{\ooo}$.
\end{corollary}
\begin{proof}
Choosing   $C=C^{\sss} \pm {C}^{\ooo}$ in \eqref{gSB-1}, we get for $t=0$
\begin{eqnarray}\label{sbcsa}
1-\mathbb{E}\left[z^{\cT} e^{i {s}\cdot ({C}_{\cT}^{\sss} \pm {C}_\cT^{\ooo}) } \right] &=&
\exp \left(-\sum_{n=1}^{\infty} \frac{z^{n}}{n} \int_{\left\{S_{n} >0 \right\}} e^{i {s} \cdot ({C}_{n}^{\sss} \pm {C}_n^{\ooo})} d \mathbb{P}\right),
\end{eqnarray}
On the other  hand, using the assumption that $(S_n,C_n^\sss,C_n^\ooo)\stackrel{\text{d}}{=}(-S_n,C_n^\sss,-C_n^\ooo)$ for all $n\in\N$, 
we also  have
\begin{equation}\label{usesym}
 \int_{\left\{S_{n} >0 \right\}} e^{i {s} \cdot { ({C}_{n}^{\sss} - {C}_n^{\ooo})}}\,\, d \mathbb{P}=  \int_{\left\{S_{n} < 0 \right\}} e^{i {s} \cdot { ({C}_{n}^{\sss} +  {C}_n^{\ooo})}}\,\, d \mathbb{P}\,.
\end{equation}

Putting together Eqs.\ \eqref{sbcsa} and \eqref{usesym}, and using the i.i.d.\ assumption about the increments  $\eta$ of ${C}^{\sss} +  {C}^{\ooo}$,  we  have
\begin{eqnarray}
&\left(1-\mathbb{E}\left[z^{\cT} e^{i {s}\cdot ({C}_{\cT}^{\sss} +{C}_\cT^{\ooo}) } \right] \right) \left(1-\mathbb{E}\left[z^{\cT} e^{ i {s}\cdot  ({C}_{\cT}^{\sss} -{C}_\cT^{\ooo})}\right] \right)
\\
&\qquad\qquad\qquad
=(1-z\phi_{{\eta_1}}({s}) )  \exp \left(\sum_{n=1}^{\infty} \frac{z^{n}}{n} \int_{\left\{ S_{n} = 0\right\}} e^{i {s} \cdot {({C}_{n}^{\sss} +{C}_n^{\ooo})}} d \mathbb{P}\right)\,,\nonumber
\end{eqnarray}
that in view of Eq.\ \eqref{phi} concludes the proof.
\end{proof}
In the \sss-symmetric case,  Eq. \eqref{mix_cor} allows to obtain an explicit representation of  
the characteristic function of the first-ladder cost.

\begin{corollary}\label{symmetric}
If $(S,C)$ is \emph{\sss-symmetric} with i.i.d.\ increments, $s \in \R^{\ell}$  and $z\in(0,1)$, then
\begin{equation}
\mathbb{E}\left[z^{\cT} e^{i {s}\cdot {C}_{\cT}}\right]
=1-\sqrt{1-z \phi_{{\eta_1}}({s})}\F(z,{s})\,.
\end{equation}
\end{corollary}
\begin{proof}
The proof follows directly from (\ref{mix_cor}) by setting 
$C^{\ooo} \equiv 0$ and solving the resulting second-order equation.
Note that the solution with the $+$ sign in front of the square root 
must be discarded, as it is incompatible with the property
$|\mathbb{E}\left[z^{\cT} e^{i {s}\cdot {C}_{\cT}}\right]| \le 1$.
\end{proof}

\begin{remark}\label{symmetric and continuous}
If  $(S,C)$  is \sss-symmetric  with i.i.d.\ increments, and the control process $S$ has a continuous distribution, then  $\F(z,s)=1$ given that $\P\{S_n=0\}=0$ and thus
\begin{equation}\label{SA-generale}
\mathbb{E}\left[z^{\cT} e^{i {s}\cdot {C}_{\cT}}\right]=1-\sqrt{1-z  \phi_{{\eta_1}}({s}) }\,,
\end{equation}
leading to the behavior  of ${C}_{\cT}$ stated in \cite{artuso}, where a combinatorial proof of this result has been provided.
Notice  that the dependence of this joint generating function on
the random walk distribution comes only through the costs, that in general
depend on $S$.
The presence of  a discrete jump distribution  yields  a correction term $\F(z,{s})$ that instead explicitly depends on the random walk,
 as already underlined in \cite[\S\,17.E1, Eq.\,(6)]{Sp76Book}.
Observe also that the identity \eqref{SA-generale} generalizes the classical Sparre-Andersen identity, which is recovered for ${s}={0}$.
\end{remark}

The generalized Spitzer-Baxter identities stated in Theorem \ref{GSB},
together with Corollaries \ref{cor-mix} and \ref{symmetric},
provide the key element to identify  the law of the first-ladder quantities
involved in it  (see e.g. \cite{Chung2001} for the classical treatment).

While the laws of $\cT$ and $S_\cT$ are  well known  under quite
general hypotheses on the random walk $S$ with i.i.d.\ increments
(see \cite{DG93} and references therein),
the focus will rather be given to the first-ladder cost.

From now on, in order to derive  the asymptotic distribution 
of $C_\cT$, we will work under the following assumptions: 
\begin{itemize}
\item[a.] the joint process $(S,C)$ is $\sss$-symmetric or $\ooo$-symmetric; 
\item[b.] the increments $\xi_k$'s of $S$ are discrete or absolutely continuous random variables.
\end{itemize}
Let us stress that, under the above assumption b., 
the support of the function
$ \F(z,{s})$  can be extended to include $z=1$ by setting,
$\forall s\in\R^{\ell}$,
$$\displaystyle\F(1,{s})=\lim_{z\to 1^-} \F(z,{s})\,.$$ 
This is trivial in the case that
$\xi_k$'s have an absolutely continuous distribution 
(see also Remark \ref{rem-phi}).
On the other hand, if the $\xi_k$'s are discrete random variables, being $S$ a symmetric random walk and using the fact that the first-ladder time $\mathcal{T}$ is a.s. finite,
one gets explicitly (see also~\cite[\S\,17.E1, page 185]{Sp76Book}) 
$$
\F(1,0)= \exp\left( \sum_{n=1}^\infty \frac{\P(S_n=0)}{2n}\right)\,,
$$
which is finite since $\P(S_n=0)\le C n^{-1/2}$, for some $C>0$.
In particular, the statement of Corollary \ref{cor-mix} holds true also for $z=1$ and any $s\in\mathbb{R}^{\ell}$.

The next result is organized into 
four distinct cases depending on whether the value 
of the mean $\mathbb{E}[\eta_1]$ is finite but 
nonzero, infinite, including the subcase associated 
with the Cauchy distribution, or zero.

\begin{proposition}\label{prop-lengthRW}
Assume that $(S,C)$ is \sss-symmetric with i.i.d.\ increments $(\xi_k, \eta_k)_{k\in\N}$.  If $\eta_1$ is in the normal basin of attraction of a $\gamma$-stable law, then $C_{\cT}$ is in the basin of attraction of a $\hat{\g}/2$-stable law. More precisely
\begin{itemize}
\item[(A)] If $\phi_{\eta_1}(s) = 1+ i \n s+ o(s)$ for  $s\to0^+$, with
$\n > 0$ real and  finite (similarly if $ \n<0$), then as $x\to\infty$
\begin{align}\label{costRW1}
\P(C_{\cT} > x )\sim \sqrt{\frac{\n }{\pi}}\F(1,0) x^{-1/2}\,,\quad \P(C_{\cT}<- x)=o(x^{-1/2})\,,
\end{align}
\item[(B)]  If $\phi_{\eta_1}(s) = 1-c_1 s^{\g}+o(s^{\g})$,
for  $s\to0^+$, $\g\in (0,1]$ and $c_1\in\mathbbm{C}$ a complex constant with $\Re(c_1)>0$,  then
as $x\to\infty$
 \be\label{costoRW2}
\P(C_{\cT} > x) \sim
\frac{Cp_+}{\G\left(1- \g/ 2\right)}\,x^{-\g/2}\,,\quad\P(C_{\cT} <-x) \sim
\frac{Cp_-}{\G\left(1- \g/ 2\right)}\,x^{-\g/2}\,,
\ee
where
\be
C=\frac{\F(1,0)}{\cos(\pi\g/4)}[\Re(c_1)^2+\Im(c_1)^2]^{1/4}\cos\left(\frac 1 2\arctan\left( \frac{\Im(c_1)}{\Re(c_1)}\right) \right)\nonumber\,,
\ee
and
\be
p_+=1-p_-=\frac 1 2\left( 1-\frac{\sin\left(\frac 1 2\arctan\left( \frac{\Im(c_1)}{\Re(c_1)}\right) \right)}{\cos\left(\frac 1 2\arctan\left( \frac{\Im(c_1)}{\Re(c_1)}\right) \right)\tan\left(\frac{\pi\g}4\right)}\right)\in[0,1]\nonumber\,.
\ee
If $p_+=0$ or $p_-=0$, then we interpret \eqref{costoRW2} as $o(x^{-\g/2})$.
\item[(C)] If $\phi_{\eta_1}(s) = 1+ic_2 s\log(1/s)+o(s\log(1/s))$,
with  $s\to0^+$ and $c_2\in\mathbbm{R}$ a positive constant (similarly for $c_2<0$), then  as $x\to\infty$
\begin{align}\label{costRW3}
\P(C_{\cT} > x) \sim \sqrt{\frac{c_2 \log(x)}{\pi}}\F(1,0)x^{-1/2}\,,\quad \P(C_{\cT} <- x)=o\left(\frac{\sqrt{\log(x)}}{x^{1/2}}\right)\,,
\end{align}
\item[(D)] If $\phi_{\eta_1}(s) = 1-c_3 s^{\g}+o(s^{\g})$,
with  $s\to0^+$, $\g\in (1,2]$ and $c_3\in\mathbbm{R}^+$ a positive constant,
  then as $x\to\infty$
 \be\label{costoRW4}
\P(C_{\cT} > +x) \sim\P(C_{\cT} <-x) \sim
\frac{C}{2\G\left(1- \g/ 2\right)}\,x^{-\g/2}\,,
\ee
where
\be
C=\F(1,0)\sqrt{c_3}\begin{cases}\displaystyle
1/\cos(\pi\g/4)\,, \qquad &\mbox{ if }\quad \g\in(1,2)\nonumber\,,\\
\displaystyle
2/\pi\,, \qquad &\mbox{ if }\quad \g=2\,.
\end{cases}
\ee
\end{itemize}
\end{proposition}
\vspace{0.5cm}
\begin{proof}
From Corollary \ref{symmetric}, we have
$$\E\left[e^{i s C_{\cT}}\right]= 1-\sqrt{1-\phi_{\eta_1}(s)} \F(1,s)\,. $$
Since the result depends solely on the behavior of the characteristic function $\phi_{\eta_1}$  around $0$, or equivalently on the tail distributions of $\eta_1$, the tail asymptotic of $C_{\cT}$ can be  readily determined via Tauberian theorems  (see e.g. \cite[\S~8.1.4]{bingham} and references therein, and refer to Appendix~\ref{app-stable}).
 By way of illustration, let us explicitly derive \eqref{costRW1}. By inserting $\phi_{\eta_1}(s) = 1+ i \n s+ o(s)$, we can write
\begin{align*}
\E\left[e^{isC_\cT}\right]&=1-\sqrt{-i\nu} \,\F(1,0)\,s^{1/2}+o(s^{1/2})\\
&=1-\sqrt{|\n|}\,\F(1,0)\,e^{-i\,\sgn(\n)\frac\pi 4}\,s^{1/2}+o(s^{1/2})\,,
\end{align*}
given that only one of the two complex square roots of $-i\nu$ satisfies the constraint for characteristic functions $|\E\left[e^{isC_\cT}\right]|\leq 1$. As a consequence, we can conclude that
\begin{equation*}
\P(C_{\cT} > x)\sim\frac c{\sqrt{\pi}}\,p_+\, x^{-1/2}\,,\quad \P(C_{\cT} <- x)\sim\frac c{\sqrt{\pi}}\,p_-\, x^{-1/2}\,,\quad \text{as }\: x\to+ \infty\,,
\end{equation*}
where
\begin{equation*}
c\coloneqq \sqrt{|\n|}\F(1,0)\,,\qquad p_-=1-p_+\,,\qquad p_+\coloneqq \frac 1 2[1+\sgn(\n)]=\begin{cases}
1\quad &\n>0\,,\\
0\quad &\n<0\,.
\end{cases}
\end{equation*}
\end{proof}

Similarly, in the \ooo-symmetric case we have the following:
\begin{proposition}\label{prop-leapRW}
Assume that $(S,C)$ is \ooo-symmetric with i.i.d.\ increments $(\xi_k, \eta_k)_{k\in\N}$, and let $\g\in(0,2)$  such that $\phi_{\eta_1}(s)=1-c_4s^\g+o(s^\g)$
for some $c_4\in\mathbbm{R^+}$. In the above notation, it holds that
\be
\P(|C_{\cT} |> x) \sim K \cdot x^{-\g/2}\,,
\ee
where the constant is explicit $K=\F(1,0)\sqrt{c_4}/\G(1-\g/2)$ whenever  $C_\cT$ is non-negative (or non-positive).
\end{proposition}

\begin{proof}
The proof follows by setting $z=1$ in Eq.~\eqref{mix_cor}, and performing a series expansion around $s=0$ on both sides. More specifically, the ansatz $\phi_{\pm C_\cT}(s)=1-c_\pm s^\a+o(s^\a)$, with $c_\pm$ complex conjugate constants, provides $|c_+|=|c_-|=\sqrt{c_4}\F(1,0)$ and $\a=\g/2$.

If $\g\in(0,2)$, then $\a\in(0,1)$ and it turns out that $\Re(c_\pm)\neq 0$.
Furthermore, for a non-negative cost process we know that $c_\pm=ce^{\mp i \frac{\pi}2 \a} \,$ (refer to Appendix~\ref{app-stable}), which concludes the proof.
Notice that if $\g=2$ (and thus $\a=1$)  we do not know if $\Re(c_\pm)\neq 0$, and hence we can not draw any conclusions about the tail distribution of $C_\cT$.
\end{proof}

Let us stress that the above propositions remain valid under the more general assumption of a $\gamma$-stable basin of attraction. The presence of slowly varying functions can be handled without additional effort, but will not be used in our main result; see~Remark~\ref{rmk:cauchy}.

\subsubsection{Applications}
The \sss-symmetric (\ooo--symmetric) condition is fulfilled in the following situations. Consider a joint process $(S,C)$ with i.i.d.\ increments $(\xi_k, \eta_k)_{k\in\N}$
such that, for a given function $g:\R\mapsto \R^{\ell}$,
\be\label{CP}
\eta_k={g}(\xi_k)\qquad \forall k\in\N.
\ee
It is apparent that if the function $g$ is even (odd)  the joint process $(S,C)$ is \sss-symmetric (\ooo-symmetric).
As a main example, let us consider the one-dimensional cost process $C\equiv L$ defined in \eqref{costLength}, corresponding to the length of the process $S$, obtained by choosing $g(\xi_k)=| \xi_k | $.
Applying the above result
we will obtain a complete characterization of the asymptotic law
of the first-ladder length $L_\cT(S)$.
Similarly, by  choosing $g(\xi_k)=\xi_k$  we will fully characterize the asymptotic behavior of the first-ladder height (or leapover) $S_\cT$.
Both results  will be of great use in the next section, we thus state them explicitly for the ease of later reference. As a consequence of Proposition~\ref{prop-lengthRW}, we get:

\begin{corollary}\label{cor-lengthRW}
Let S have i.i.d.\ symmetric increments in the normal domain of attraction of a $\beta$-stable law. Then the first-ladder length $L_\cT(S)$ is in the normal basin of attraction of a ${\hat{\b}}/2$-stable law.  More precisely, writing $\phi_{\xi_1}(s)=1-\nu  s^{\b}+o(s^{\b})$ with $\beta\in(0,2]$, we have
\be
\P(L_{\cT}(S) > x) \sim \frac{\sqrt{C}}{\G\left(1- \hat{\b}/ 2\right)}\F(1,0)x^{-\hat{\b}/2}\,
\qquad \mbox{as } x\to\infty\,,
\ee
where
$$C\coloneqq \begin{cases}
\n/\cos(\pi\hat{\b}/2)\qquad&\mbox{if}\quad\b\in(0,1)\\
2\n/\pi\log(x)\qquad&\mbox{if}\quad\b=1\\
 \E[|\xi_1|]\qquad&\mbox{if}\quad\b\in(1,2]
\end{cases}\,.$$

\end{corollary}

\begin{remark}\label{rem: LW}
As a notable application, consider the situation in which a random time is needed to perform  a jump for the random walker $S$. In this case, the total time can  be considered   (per our notation) as a cost associated with the random walk.  In particular  (but see \cite{artuso}  for details and  physical motivations) suppose that  the time taken to perform a jump is correlated with its length.
This is indeed the case for 1D L\'evy walks, which are a continuous-time interpolation (with unit speed)  of 1D $RW$ with i.i.d.\ and heavy-tailed jumps, a.k.a. L\'evy flights \cite{ZDK}. Thus, $L_\cT(S)$  corresponds to  the first-passage time for the wait-then-jump model associated with a L\'evy walk, as mentioned in \cite{artuso}.
\end{remark}

\begin{remark}
It is worthwhile to point out that Corollary \ref{cor-lengthRW} extends and completes a previous result by Sinai (see \cite[Theorem 3]{S57}). One can easily retrace his proof in the presence of an appropriate cost, still fulfilling necessary hypotheses, in order to get the basin of attraction of $L_\cT(S)$  rather than the leapover, but under the assumption that the random variables $\xi_k$'s have stable distribution.
\end{remark}

The domain of attraction of the leapover $S_\cT$, instead, stems from standard results of fluctuation theory (see \cite{DG93} and references therein). Here is obtained by applying Proposition~\ref{prop-leapRW}:

\begin{corollary}\label{cor-leapRW}
Let S have i.i.d.\ symmetric increments in the normal domain of attraction of a $\beta$-stable law with $\phi_{\xi_1}(s)=1-\nu  s^{\b}+o(s^{\b})$ and $\beta\in(0,2)$. Then
the first-ladder height $S_\cT$ is in the normal domain of attraction of a ${{\b}}/2$-stable law.  More precisely
\be
\P(S_{\cT} > x) \sim \frac{\sqrt{\n}}{\G\left(1- \b/ 2\right)}\F(1,0)x^{-\b/2}\,
\qquad \mbox{as } x\to + \infty\,.
\ee
\end{corollary}

\medskip
Notice that the limiting case $\b=2$ is discussed in~\cite{Rog71} and~\cite[Theorem 4]{DG93}.
\\

Another helpful tool, that will be used repeatedly throughout the proof of our main result, concerns linear combinations of the random variables $L_\cT(S)$ and $S_\cT$:
\begin{lemma}\label{lemma 1}
Consider a cost process $C_\cT(S)$ such that
$C_\cT(S)= L_\cT(S)+ S_\cT$  or $C_\cT(S)=L_\cT(S)- S_\cT$.
If $\beta\in[1,2]$, then
\[
\P(C_\cT(S)>x)\sim \P(L_\cT(S)>x)\quad\text{as}\quad x \to \infty\,;
\]
If $\beta\in(0,1)$, then
\[
\P(C_\cT(S)>x)\asymp  x^{- \beta/2}\quad\text{as}\quad x \to \infty\,.
\]
\end{lemma}
\begin{proof}
We have the following cases.
\begin{enumerate}[label=(\roman*)]
\item\label{en: first item} If $C_{\cT}(S)=L_\cT(S)+S_\cT$, it is enough to directly apply Lemma~\ref{lemma cauchy}. Then we can conclude by means of  the Tauberian theorem for dominated variation~\cite[Thm.~2.10.2]{bingham} for $\beta\in(0,1)$ and~\cite[Thm.~1.7.6]{bingham} for $\beta\in[1,2]$, respectively (see also Appendix~\ref{app-stable}).
\item 
If  $C_{\cT}(S)=L_\cT(S)-S_\cT$, it is convenient to define, for any $s\ge 0$, the generating functions
\be\label{generating}
\mathcal{G}_{L_\cT(S)\pm S_\cT}(s)\coloneqq \mathbb{E}\left[ e^{-s (L_\cT(S)\pm S_\cT)}\right]. 
\ee
We can then conclude the proof exploiting \ref{en: first item} together with the analog of Corollary \ref{cor-mix}
stated for the generating function of the cost random variable. 
More specifically, as $s\to 0^+$, we can write
$$
\left(1-\mathbb{E}\left[e^{-s (L_\cT\pm S_\cT)}\right] \right)\left(1-\mathbb{E}\left[e^{-s (L_\cT\mp S_\cT)}\right] \right) 
=\left(1-\E[e^{-s(|\xi_1|\pm \xi_1)}]\right)\F^2(1,s)\,,
$$
where $\phi_{\xi_1,\eta_1}$ in~\eqref{phi} has to be meant as a double Fourier-Laplace transform in $(\xi_1,\eta_1)$. 
It is obvious that, on the left-hand side of the above display, $1-\cG_{L_\cT+S_\cT}(s)\asymp s^{\hat\b/2}$ thanks to \ref{en: first item}, possibly with the logarithmic correction $\sqrt{\log(1/s)}$ when $\beta=1$ and an exact estimate~$\sim$ for $\beta\geq 1$. On the right-hand side, instead, observing that $\E[|\xi_1|\pm \xi_1]=\E[|\xi_1|]\neq 0$, we have $1-\cG_{|\xi_1|\pm \xi_1}\F^2(1,s)\sim k\,s^{\hat\b}$ for some positive constant $k$, multiplied by $\log(1/s)$ when $\beta=1$. Consequently, we obtain upper and lower bounds (matching for $\beta\geq 1$) on the leading term of the asymptotic expansion of the generating function of $L_\cT-S_\cT$.
The desired conclusion immediately follows by applying the aforementioned Tauberian theorems to $\cG_{C_\cT}(s)\sim  \cG_{L_\cT-S_\cT}(s)$. 
\end{enumerate}
\end{proof}
\begin{remark}
In the presence of spatio-temporal correlations, as explained in  Remark~\ref{rem: LW}, notice that the cost process $C_\cT\coloneqq L_\cT(S)-S_\cT$ defined in Lemma \ref{lemma 1} corresponds to the first-passage time for the L\'evy Walk.
\end{remark}

\subsection{Results for ladder costs associated with RWRSB }
The focus of the present section is the cost process $C= C(S,\z^{\pm})$
defined in Section \ref{MeB}, and called RWRSB.
We remind that the process  $C$ collects all the  scenery values $\zeta^{\pm}_k$ corresponding to the bonds that have been crossed in every jump of $S$, taking into account also the travel direction. In particular, the random scenery creates
a dependence between the increments of $C$,
and breaks down the i.i.d.\ assumption of the generalized Spitzer-Baxter identity
stated in Theorem \ref{GSB}.

In this subsection, we will study the first-ladder costs associated with a RWRSB,
and extend the results derived in the previous subsection to this general context. This analysis will lead to Theorem \ref{MainResult1}, that provides the asymptotic distribution of $C_\cT$ under the assumption that the underlying random walk has i.i.d.\ symmetric increments.
As stressed just after Theorem \ref{MainResult1}, our main result is now stated
and proved emphasizing all the different scenarios arising as the parameters of the problem vary.

As a final observation, we underline that the extension to ladder costs
$(C_{\cT_k})_{k\in\mathbbm{N}_0}$, where  $\cT_k$ is the ladder time corresponding
 to the $k$-th maximum value reached by $S$, will be directly dealt with along the proof of the main theorem.
\newline

First of all, let us fix some notation. We consider a symmetric underlying random walk $S$ on $\Z$ with i.i.d.\ discrete increments $(\xi_k)_{k\in\mathbbm{N}}$, whose corresponding characteristic function is, for $s\to 0^+$,
\be\label{eq: charXi}
\phi_{\xi_1}(s)=1-\nu  s^{\b}+o(s^{\b})
\ee
with $\b\in(0,2)$ and $\n \in\mathbbm{R}^+$. 
Since $\xi_1$ is a symmetric random variable, whereas $|\xi_1|$ is one-sided distributed and with 
$\mathbb P[|\xi_1|>x]=2\mathbb P[\xi_1>x]$,  the characteristic function of $|\xi_1|$ is of the form
(refer to Appendix~\ref{app-stable}),
for $s\to 0^+$,
\be\label{eq: charAbsXi}
\phi_{|\xi_1|}(s)=1+\hat{\n} s^{\hat{\b}}+o(s^{\hat{\b}})\,,
\ee
where
$$
\hat{\n}=  \begin{cases}
-\n [1-i\tan(\pi \b/2)]\,,\qquad &\mbox{for}\quad \b\in (0,1)\,,\\
 -\n [1-i\frac 2 \pi \log(1/s)]\,,\qquad &\mbox{for} \quad\b=1\,,\\
i\E[|\xi_1|]\,,\qquad &\mbox{for}\quad \b \in (1,2]\,.
\end{cases}
$$

We also explicitly write the common characteristic function of the random variables $\zeta_k^{\pm}$'s that we suppose are in the normal basin of attraction of a $\g_{\pm}$-stable law respectively, with $\g_{\pm}\in (0,2]\backslash\{1\}$: for $\th \to 0^+$ 
\begin{align}\label{eq: charf}
\phi_{\z_1^{\pm} }(\th) &
&=\begin{cases}
1-c_\pm\th^{\g_\pm}+o(\th^{\g_\pm})\,,\quad &\g_{\pm}=\hat{\g}_\pm\in (0,1)\,;\:c_\pm\in\mathbbm{C}\,,\:\Re(c_\pm)>0\,,\\
1+i\m_\pm\th- c_\pm\th^{\g_\pm}+o(\th^{\g_\pm})\,,\quad & {\g}_\pm\in (1,2]\,,\hat{\g}_\pm=1 \,;  \:\m_\pm\in \mathbbm{R}\,,\: c_\pm\in \mathbbm{C}\,,\\
1\,,\quad &\g_\pm=\hat{\g}_\pm=+\infty\;\implies\; \z_1^{\pm}\equiv 0\,.
\end{cases}
\end{align}

We will also need to refer to the even part of the scenery values
$\zeta_k^0= \frac{ \zeta_k^+ +\zeta_k^-}2$, for all $k\in\Z$,
 and assume that their common  characteristic function is given by
\begin{align}\label{eq: charPf}
\phi_{\z_1^0}(\th) &
&=\begin{cases}
1-c_0\th^{\g_0}+o(\th^{\g_0})\,,\quad &\g_0=\hat{\g}_0\in (0,1)\,;\:c_0\in \mathbbm{C}\,,\:\Re(c_0)>0\,,\\
1+i\m_0\th- c_0\th^{ \g_0}+o(\th^{\g_0})\,,\quad &  {\g}_0\in (1,2]\,, \hat{\g}_0=1\,;\; \:\m_0\in \mathbbm{R}\,,\: c_0\in \mathbbm{C}\,,\\
1\,,\quad &\g_0=\hat{\g}_0=+\infty\;\implies\; \z_1^0\equiv 0\,.
\end{cases}
\end{align}

Let us discuss  the relationship between the cost exponents $\hat{\g}_\pm$ and $\hat{\g}_0$, which will be crucial for the structure of the proof. By applying Lemma~\ref{lemma cauchy}, it is easy to verify that
\begin{itemize}
\item  if $\hat{\g}_+ \neq  \hat{\g}_-$, we have $\hat{\g}_0=\min\{\hat{\g}_+,\hat{\g}_-\}$
\item  if $\hat{\g}_+=\hat{\g}_-$:
\begin{itemize}
\item $\hat{\g}_0=\hat{\g}_+=\hat{\g}_-$, or
\item $\hat{\g}_0 > \hat{\g}_+=\hat{\g}_-\,$, including two possible cases:
\begin{enumerate}[label=(\alph*)]
\item $(0,1]\ni\hat{\g}_0 > \hat{\g}_+=\hat{\g}_-\in (0,1)\, \implies$\,
$\z_1^+=-\z_1^-+h(\z_1)$ with $h(\z_1)\neq 0$ and
$\hat{\g}_0\equiv\hat{\g}_{h(\z_1)}\,$;
\item $\hat{\g}_0=+\infty \implies \z_1^0\equiv 0\,.$
\end{enumerate}
\end{itemize}
\end{itemize}
\hspace{0.8cm}

Finally,
let $(\cT_{n})_{n\geq 0}$ be the sequence of ladder times of the control process $S$, namely the consecutive times when the random walk attains a new maximum value.  Formally, they are recursively defined by
$$\cT_0=0\,,\quad \cT_n:=\min\{k> \cT_ {n-1}\,:\, S_k > S_{\cT_{n-1}}\}
\quad \forall n\in\N\,,$$
so that $\cT_1\equiv \cT$.
Notice that by the Markov property, they give rise to a renewal process.
\\

To state the main theorem, let us recall the notation introduced in \eqref{notation}, 
with $\r_+\coloneqq \hat{\g}_+\,\b/2$ and $\r_0 \coloneqq \hat{\g}_0\, \hat{\b}/2$.
\begin{theorem}\label{MainResultExtended}
Let $C$  be the cost process defined in~\eqref{cost-RWRE}, and $(C_{\cT_n})_{n\geq 0}$ the corresponding ladder cost process.
Suppose that the underlying increments $\xi_k$'s satisfy~\eqref{eq: charXi}, and that $\z^+,\z^0$ are i.i.d.\ sequences of  non-negative  (similarly for non-positive) random variables satisfying~\eqref{eq: charf} and~\eqref{eq: charPf} respectively. Then, for all $n\in\N$, the following results hold as $x\to\infty$: 
 \begin{itemize}
\item If $\r_+<\r_0\,$, there exists an explicit constant $K\in\mathbbm{R}^+$  such that
$$  \P(C_{\cT_n}>x)\sim K\cdot n\cdot x^{-\r_+}\,.$$

\item If $\r_+\ge \r_0$, there exists an explicit slowly varying function 
$K_{up}(x)$ such that
\[
\P(C_{\cT_n}>x) \leq K_{up}(x)\cdot n\cdot x^{-\r_0}\,,
\]
where $K_{up}(x)\equiv k_{up}\in\mathbbm{R}^+$ if $\b\neq 1$,  
and $K_{up}(x)= k_{up}\sqrt{\log(x)}$ if $\b=1$.\\
Moreover, if $\beta\geq 1$, there exists an explicit slowly varying function $K_{low}(x)$ such that 
\be\label{eq:LB}
\P(C_{\cT_n}>x) \geq  K_{low}(x)\cdot n\cdot x^{- \min\left\{\rho_+,\hat\g_0,1/2\right\}}\,,
\ee
where $K_{low}(x)\equiv k_{low}\in\mathbbm{R}^+\,$ unless $\b=1$ with tail exponent $1/2$ in~\eqref{eq:LB}, or 
if $\hat\gamma_0=1/2$, for which logarithmic corrections appear. 

\noindent
In particular, when $\g_0\in(1,2]$, the tail exponent in the lower bound is precisely 
$1/2\equiv\rho_0$ and thus matches with that of the upper bound.
\end{itemize}
\end{theorem}

\medskip
\begin{remark}\label{rmk:cauchy}
We emphasize that the lower bound in Eq.~\eqref{eq:LB} is the only estimate where the assumptions on the normal domain of attraction for $\xi$, $\zeta^+$, and $\zeta^0$ are strictly necessary, along with the requirement that $\beta \geq 1$.
As will become clear in the proof, these stronger conditions are due to the application of Lemma~\ref{lemma product}, which requires the normal domain of attraction for the variables in the game,
and Lemma~\ref{lemma 1}, which provides different information depending on the value of $\beta\in(0,2]$.
Although these conditions may seem predominantly technical, it was not possible to circumvent these assumptions.
We also stress that the limiting cases 
$\g_+,\g_0 = 1$ have been excluded to simplify the computations, but they can be handled using the same argument given in the proof, with careful attention to the additional logarithmic corrections.
\end{remark}

\subsubsection{Preliminary tools}
Following \cite{blp}, it is convenient to introduce the family of random variables
$\mathcal N_n(k)$, for $k\in\Z$ and $n\in\N$, called \emph{local times on the bonds}
of the random walk $S$, and given by
\be
  \cN_n(k)\coloneqq \#\{j \in \{1,\ldots,n\} \,:\, [k-1,k] \subseteq
  [S_{j-1}, S_j]\} \, ,
\ee
where the notation $[a, b]$ denotes the closed interval
between the real numbers $a$ and $b$, irrespective of their order.
In other words, $\mathcal N_n(k)$ is the number of times that
the walk $S$ travels the bond $[k-1, k]$
and in turn can be split into $\cN_n(k)=\cN_n^-(k)+\cN_n^+(k)$ where
\bea
\cN_n^-(k) &\coloneqq &\#\{j \in \{1,\ldots,n-1\} \,:\, S_j\geq k\,, S_{j+1}\leq k-1\} \, ,\nonumber\\
\cN_n^+(k) &\coloneqq& \#\{j \in \{1,\ldots,n-1\} \,:\, S_{j+1}\geq k\,, S_j\leq k-1\} \, ,\nonumber
\eea
denote the number of crossings of $[k-1,k]$, respectively, from right to left and from left to right. In the following, it will be useful to express the first-ladder height and length of the process $S$ in terms of local times. An easy check shows that
\begin{subequations}
\bea
\mbox{if} \; k\leq 0\,, \quad & \cN_\cT^+(k)=\cN_\cT^-(k)=\frac{\cN_\cT(k)}2\,,\quad &\sum_{k\leq 0} \cN_\cT(k)=L_\cT(S)-S_\cT\,,\\
\mbox{if} \; k>0\,, \quad & \cN_\cT(k)=\begin{cases}
1 \; &k\leq S_\cT\\
0\; & k>S_\cT
\end{cases}\,,\quad &\sum_{k> 0} \cN_\cT(k)=S_\cT\,.
\eea
\label{HeLinLocalTimes}
\end{subequations}
Other useful probabilistic results are postponed to Appendix \ref{app-prob}.

\subsubsection{Proof of  Theorem \ref{MainResult1}}
We provide the proof of Theorem \ref{MainResultExtended}, which contains the statements  of  Theorem \ref{MainResult1} in an extended version.

\begin{proof}[Proof of Theorem \ref{MainResultExtended}]
As underlined in \cite{blp}, the introduction of the local times $\mathcal N_n(k)$'s
provides an interpretation of the collision times $(T_n)_{n\in\N_0}$ (defined in \eqref{collisiontime})
 as a random walk in random scenery on bonds.
More generally, the cost process of the form \eqref{cost-RWRE} satisfies the following identity
 \begin{align}
  C_{\cT}
&= \sum_{k\in\Z} \left[\cN_\cT^+(k) \z_k^++\cN_\cT^-(k)\z_k^- \right]\label{colltime_0}\\
&= \sum_{k\leq 0}\cN_\cT(k) \z_k^0+\sum_{k>0}\cN_\cT(k)\z_k^+\,.\label{colltime}
  \end{align}
Since the local times are functions of $S$ only, it turns out that, given $S$, $C_\cT$ is a sum of independent random variables.

Using \eqref{colltime}, we can rearrange the terms inside the characteristic function of $C_\cT$ and get, for $s\in \R$,
\begin{align}
 \mathbb{E} \left[
 e^{ i s C_{\cT}}\right]\nonumber
& =  \mathbb{E} \left[
\exp \left(\: i s \sum_{k=1}^{\cT} \eta_k    \:\right)  \right]
 =  \mathbb{E} \left[
 e^{ i s\left[\sum_{k\leq 0}\cN_\cT(k) \z_k^0+\sum_{k>0}\cN_\cT(k)\z_k^+\right ]}
 \right]\\ 
& =  \mathbb{E} \left[ \left. \mathbb{E} \left[
\exp \left(\: i s \sum_{k\leq 0}\cN_\cT(k) \z_k^0\:\right)\exp \left(\: i s\sum_{k>0}\cN_\cT(k)\z_k^+  \:\right)
 \right| S \right]\right]\nonumber\\
 & = \mathbb{E} \left[
\prod_{k \leq 0} \phi_{\z_1^0}(s \cN_{\cT}(k)) \cdot \left(\phi_{\z_1^+}(s) \right) ^{S_\cT}
\right] \label{splitting}\,,
\end{align}
where in the last line we used the conditional independence mentioned above.

Hereinafter we will move to the generating function formalism, which is justified by  the additional hypothesis $\z_1^+,\z_1^0\geq 0$.
In particular, we will establish upper and lower bounds for the generating function
$$\mathcal{G}_{C_\cT}(s)\coloneqq \E[e^{-sC_\cT}]\,,$$
using the analog of~\eqref{splitting} for generating functions. This will allow us to determine, respectively, lower and upper bounds for the tail of the random variable $C_\cT$ by means of the Tauberian theorem for dominated variation~\cite[Thm.~2.10.2]{bingham}.
We will split our analysis depending on whether the value of $\g_0$ is infinite or finite.  Accordingly to the specific regimes of values of $\b$, $\g_+$ and $\g_0$, 
the results collected in Theorem \ref{MainResultExtended} will be derived.

\vspace*{0.5cm}
\paragraph{\bf Case  $\hat\g_0=+\infty \:\implies\: \r_+<\r_0\,$} Here we get a direct result, since local times disappear. Explicitly, recalling that -- according to Eq.~\eqref{eq: charf} -- for $s\to 0^+$
\be\label{eq: genf}
\mathcal{G}_{\z_1^+}(s)=1-\tilde \m_{+}s^{\hat\g_+}+o(s^{\hat\g_+})\quad \mbox{with}\quad \tilde\m_{+}=\begin{cases}
 \frac{\Re(c_+)}{\cos(\pi\hat{\g}_+/2)}\qquad& \mbox{if}\quad\hat\g_+\in(0,1)\,,\\
\m_+\qquad&\mbox{if}\quad\hat\g_+=1\,,
\end{cases}
\ee
we have
$$
\cG_{ C_{\cT}}(s)
= \E\left[ \left(\cG_{\z_1^+}(s) \right) ^{S_\cT}\right]=\E\left[e^{-(\tilde \mu_++o(1))s^{\hat\gamma_+}S_\cT}\right]\sim\mathcal{G}_{S_\cT}(\tilde\m_{+}s^{\hat\g_+}) \quad\text{as}\quad s\to0^+.
$$
Hence, by applying the Tauberian theorem~\cite[Thm.~1.7.6]{bingham} as in Corollary~\ref{cor-leapRW}, we immediately obtain
\be\label{eq:rho+}
\mathbb{P}(C_\cT>x)\sim K\cdot x^{-\hat\g_+\b/2}\quad\text{as}\quad x \to\infty\,,
\ee
with $K= \frac{\sqrt{\nu}\tilde \mu_+^{\b/2}\,\F(1,0)}{\G(1-\hat\g\b/2)} $.

\vspace*{0.5cm}
\paragraph{\bf Case $\hat\g_0<\infty$ (equiv. $0<\hat\g_0\leq 1$)} We start from Eq.~\eqref{splitting}, that is
\be\label{splitting - gen} 
 \mathcal{G}_{C_\cT}(s)=\E[e^{-sC_\cT}]= \mathbb{E} \left[Z_1\cdot Z_2 \right] \,, \qquad s\geq 0\,,
\ee
with
\[
Z_1\coloneqq  \prod_{k \leq 0} \mathcal{G}_{\z_1^0}(s \cN_{\cT}(k))\qquad\text{and}\qquad Z_2\coloneqq\left(\mathcal{G}_{\z_1^+}(s) \right) ^{S_\cT}\,.
\]
We intend to apply Lemma~\ref{lemma cauchy}. Using Eq.~\eqref{eq: genf}, observe that as in the previous paragraph $\E[Z_2]$ is trivially asymptotically equivalent to $\mathcal{G}_{S_\cT}(\tilde\m_{+}s^{\hat\g_+})=1-c_2s^{\g_2}+o(s^{\g_2})$, with $\g_2=\hat\g_+\b/2$ and $c_2$ given by Corollary~\ref{cor-leapRW}, so we have to focus our efforts on the random variable $Z_1$. By the moment monotonicity, given $0<p\leq r$ we know that
\[
 \left(\E[e^{-s\zeta_1^0 p}]\right)^{1/p}\leq \left(\E[e^{-s\zeta_1^0 r}]\right)^{1/r}\,,
\]
which can be used to obtain a lower bound ($p=1$ and $r=\cN_\cT(k)$) and an upper bound ($p=\cN_\cT(k)$ and $r=L_\cT-S_\cT$) on the random variable $Z_1$:
\be\label{bound Z1}
\E[Z_1^\text{low}\cdot Z_2]\leq \cG_{C_\cT}(s)\leq \E[Z_1^\text{up}\cdot Z_2],\quad \text{where} \quad\begin{cases}
Z_1^\text{low}\coloneqq \left (\cG_{\zeta_1^0}(s)\right)^{L_\cT-S_\cT},\\
 Z_1^\text{up}\coloneqq \cG_{\zeta_1^0}(s(L_\cT-S_\cT))\,.
\end{cases}
\ee

\vspace*{0.2cm}
\subparagraph{\it Lower bound:} We can replace the generating function in $Z_1^\text{low}$ with its expansion for $s\to 0^+$ according to \eqref{eq: charPf}
\[
\mathcal{G}_{\z_1^0}(s)=1-\tilde \m_0 s^{\hat\g_0}+o(s^{\hat\g_0})\quad \mbox{with}\quad \tilde\m_{0}=\begin{cases}
 \frac{\Re(c_0)}{\cos(\pi\hat{\g}_0/2)}\qquad& \mbox{if}\quad\hat\g_0\in(0,1)\,,\\
\m_0\qquad&\mbox{if}\quad\hat\g_0=1\,,
\end{cases}
\]
and obtain
$
\E[Z_1^\text{low}]\sim \cG_{L_\cT-S_\cT}(\tilde \mu_0 s^{\hat\g_0})\,.
$
By  Lemma~\ref{lemma 1}, when $\beta\in[1,2)$ we therefore have
$
\E[Z_1^\text{low}]=1-c_1s^{\g_1}+o(s^{\g_1})
$
with $\gamma_1=\hat\g_0\hat\b/2$ and an additional logarithmic factor $\sqrt{\log(1/s)}$ if $\b\equiv\hat\b=1$; see also Corollary~\ref{cor-lengthRW}.
Using the estimates in Lemma~\ref{lemma cauchy}, we can finally write (with the same logarithmic correction) 
$$
\mathcal{G}_{C_\cT}(s)\geq \E[Z_1^\text{low}]\E[Z_2]-\sqrt{\var(Z_1^\text{low})\var(Z_2)}
=1-c_{low}s^{\min\{\r_0,\r_+\}}+o(s^{\min\{\r_0,\r_+\}})\,,
$$
given that $\sqrt{\var(Z_1^\text{low})\var(Z_2)}=\sqrt{(2-2^{\r_0})(2-2^{\r_+})c_1c_2}s^{\frac{\r_0+\r_+}2}.$

\noindent
Similarly, when $\beta\in(0,1)$ we have $1-\E[Z_1^\text{low}]\asymp s^{\g_1}$: denoting by $c_1^\pm$ the upper and lower constants coming from the dominated variation, it is sufficient to replace $c_1$ with $c_1^+$, and $\var(Z_1^{low})\leq 2c_1^+ s^{\rho_0}+o(s^{\rho_0})$. 

\vspace*{0.5cm}
\subparagraph{\it Upper bound:} Since by assumption $\zeta_1^0$ and $L_\cT-S_\cT$ are independent random variables, the law of total expectation gives
\[
\E[Z_1^\text{up}]=\cG_{\zeta_1^0\cdot(L_\cT-S_\cT)}(s).
\]
We now split the analysis according to the value of $\hat\g_0$, focusing on the case $\hat \beta=1$. 

If $\hat\g_0=1\,$ (equiv. $\g_0\in (1,2]$), by exploiting Lemma~\ref{lemma product} we can affirm that the asymptotic behavior of $Z_1^\text{up}$ is ruled by $L_\cT-S_\cT$, that is the random variable with slower tail decay:  
\[
\E[Z_1^\text{up}]= 1-c_1 s^{\g_1}+o(s^{\g_1}),
\]
with $\g_1=\hat\b/2$, except for the case $\b\equiv\hat\b=1$ where the constant $c_1$ is replaced by a slowly varying function $c_1 \sqrt{\log(1/s)}$.
Relying again on Lemma~\ref{lemma cauchy}, we get
$$
\mathcal{G}_{C_\cT}(s)\leq \E[Z_1^\text{up}]\E[Z_2]+\sqrt{\var(Z_1^\text{up})\var(Z_2)}
=1-c_{up}s^{\min\{\r_0,\r_+\}}+o(s^{\min\{\r_0,\r_+\}})\,,
$$
with the same (eventual) logarithmic correction as before. In particular, observe that when $\r_+<\r_0$, combining upper and lower bounds, we obtain $\cG_{C_\cT}(s)\sim \E[Z_2]$. Hence we simply recover~\eqref{eq:rho+}.

If $\hat\g_0\equiv\g_0\in (0,1)$, minor changes are required: by Lemma~\ref{lemma product} the tail decay of $Z_1^{up}$ is now determined by $\g_1=\min\{\hat\g_0,\hat\b/2\}$, with a logarithmic correction
\[
\begin{cases}
\sqrt{\log(1/s)}\quad &\text{if}\quad \g_0>1/2,\: \beta\equiv \hat \b=1,\\
\log(1/s)\quad &\text{if}\quad  \g_0=1/2,\: \beta>1,\\
\log^{3/2}(1/s)\quad & \text{if}\quad \g_0=1/2,\: \beta\equiv \hat \b=1.
\end{cases}
\]

\vspace*{0.5cm}

As a final point, we generalize all these asymptotic results to the law of ladder costs.
Let $(\cT_{n})_{n\geq 0}$ be the sequence of ladder times of the random walk  $S$ defined in the introduction. Notice that, for all $n\geq 1$, the equality \eqref{colltime_0} still holds by replacing $\cT$ with $\cT_n$, together with the identity $L_{\cT_n}(S)= \sum_{k\in\Z}\cN_{\cT_{n}}(k)$. Moreover, $S_{\cT_n}=\sum_{k>0}\mathbbm{1}_{(0\,,\,S_{\cT_n}]}(k)\,$.
In order to recast the characteristic or generating function of the cost $C_{\cT_n}$ in a convenient manner, we need to introduce a further definition concerning the local times. Let $\cN_{(t_0,t_f]}(k)$ be the number of crossings of $[k-1,k]$ observed in a specified time window $(t_0,t_f]$, that is
\be
  \cN_{(t_0,t_f]}(k)\coloneqq \#\{j \in \{t_0+1,\ldots,t_f\} \,:\, [k-1,k] \subseteq
  [S_{j-1}, S_j]\} \, .
\ee
Hence we get
\begin{align*}
C_{\cT_n}
&=\sum_{k\in\mathbbm{Z}}\left[\cN_{\cT_n}^+(k)\z_k^++\cN_{\cT_n}^-(k)\z_k^-\right ]\nonumber\\
& =  \sum_{k\leq 0} \cN_{(0,\cT_1]}(k)\z_k^0+\sum_{k\leq S_{\cT_1}} \cN_{(\cT_1,\cT_2]}(k)\z_k^0+\dots+\sum_{k\leq S_{\cT_{n-1}}} \cN_{(\cT_{n-1},\cT_n]}(k)\z_k^0\nonumber\\
&\hspace{1.5cm} +\sum_{k>0}\cN_{(0,\cT_1]}(k)\z_k^++\sum_{k>S_{\cT_1}}\cN_{(\cT_1,\cT_2]}(k)\z_k^++\dots +\sum_{k>S_{\cT_{n-1}}}\cN_{(\cT_{n-1},\cT_n]}(k)\z_k^+\\
&=\sum_{k\leq 0} \cN_{\cT_n}(k)\z_k^0+\sum_{k\in(0, S_{\cT_1}]} \cN_{(\cT_1,\cT_n]}(k)\z_k^0+\dots+\sum_{k\in(S_{\cT_{n-2}},S_{\cT_{n-1}}]} \cN_{(\cT_{n-1},\cT_n]}(k)\z_k^0\\
&\hspace{1.5cm} +\sum_{k>0}\mathbbm{1}_{(0,S_{\cT_n}]}(k)\z_k^+\,,
\end{align*}
and also
\begin{align*}
 \mathcal{G}_{C_{\cT_n}}(s)&\coloneqq\E \left[ e^{ - s C_{\cT_n}}\right]\nonumber
=  \mathbb{E} \left[ \left. \E \left[ e^{ - s\sum_{k\in\mathbbm{Z}}\left[\cN_{\cT_n}^+(k)\z_k^++\cN_{\cT_n}^-(k)\z_k^-\right ]}
 \right| S \right]\right]\nonumber\\
 & = \mathbb{E} \left[
\prod_{k \leq 0} \mathcal{G}_{\z^0_1}(s \cN_{\cT_n}(k))\prod_{k \in(0,S_{\cT_1}]} \mathcal{G}_{\z^0_1}(s \cN_{(\cT_1,\cT_n]}(k))\dots \right.\\
&\hspace{3cm} \left.\dots \prod_{k \in(S_{\cT_{n-2}},S_{\cT_{n-1}}]} \mathcal{G}_{\z^0_1}(s \cN_{(\cT_{n-1},\cT_n]}(k)) \cdot \left(\mathcal{G}_{\z^+_1}(s) \right) ^{S_{\cT_n}}
\right] \label{splitting}\,.
\end{align*}
Then, when we  consider upper and lower bounds for $\mathcal{G}_{C_{\cT_n}}(s)$ (see Eq.~\eqref{bound Z1}), the relevant quantities to deal with are $S_{\cT_n}$ and
$$
\sum_{k \leq 0} \cN_{\cT_n}(k)+\sum_{k \in(0,S_{\cT_1}]}  \cN_{(\cT_1,\cT_n]}(k)+\dots+\sum_{k \in(S_{\cT_{n-2}},S_{\cT_{n-1}}]} \cN_{(\cT_{n-1},\cT_n]}(k)=L_{\cT_n}-S_{\cT_n}\,.
$$
More precisely, we are interested in their generating functions. Due to the renewal structure of the processes $(S_{\cT_n})_{n\geq 0}$ and $(L_{\cT_n}(S))_{n\geq 0}$, the ladder random variables can be seen as the sum of $n$ i.i.d.\ first-ladder quantities. Thus, the previous results can be immediately generalized: we just have to introduce a multiplicative factor $n$, which stems from the factorization of the expectations, in front of the slowly varying functions $K$, $K_{up}(x)$ and $K_{low}(x)\,$.
\end{proof}
\vspace*{0.2cm}

\subsection{Results for the random walks in random media Y}
Consider the random walk on random medium $Y$ defined in Eq.~\eqref{Y},
under the hypothesis that the underlying random walk $S$
has symmetric i.i.d.\ increments $(\xi_k)_{k\in\N}$.
As already mentioned, the  first-ladder height $Y_{\cT}$ and the first-ladder length $L_{\cT}(Y)$ can be equivalently interpreted as first-ladder costs expressed as RWRSB with appropriate sceneries. We then use our main Theorem \ref{MainResultExtended}, together with some simplifications that occur in this setting, to prove Corollaries \ref{leapoverY-app}, \ref{lengthY}, and \ref{gLG}.

In the following, we will refer to the notation introduced in previous sections, except for Eq.~\eqref{eq: charf} that can be slightly simplified by considering only
\begin{align*}
\phi_{\z_1 }(\th) &
=\begin{cases}
1-c e^{-i\frac\pi 2 \g}\th^{\g}+o(\th^{\g})\,,\quad &\g=\hat{\g}\in (0,1)\,;\:c\in\mathbbm{R}^+\,,\\
1+i\m\th+o(\th)\,,\quad & {\g}\in (1,2]\,,\hat{\g}=1 \,;  \:\m\in \mathbbm{R}^+\,.
\end{cases}
\end{align*}

\subsubsection{ First-ladder height $Y_{\cT}$. }
\begin{proof}[Proof of Corollary \ref{leapoverY-app}]
Recall that $Y_{\cT}$ can be seen as the value at time $\cT$ of a
RWRSB driven by $S$ and with scenery  $\z^+=-\z^-=\z$.
Notice that with this choice
$\hat{\g}\equiv \hat{\g}_+=\hat{\g}_-$, and $\hat{\g}_0=+\infty\,$ since $\zeta_1^0\equiv 0$, implying that $\r_+<\r_0\,$. 

\noindent
Therefore Theorem \ref{MainResultExtended} applies and we get:

\begin{itemize}
\item If $\g\in(1,2]$, which means $\hat{\g}=1$, then $\E[e^{-sY_\cT}]=\E[e^{-s(\m+o(1))S_\cT}]$ and
$Y_\cT$ is in the normal basin of attraction of a stable law with parameter $\b/2\,$. Explicitly, by applying Corollary~\ref{cor-leapRW}, we get
\be\label{leapoverY-app-case1}
\P(Y_\cT> x) \sim \frac{\sqrt{\nu}\mu^{\b/2}}{\G(1-\b/2)}\, \F(1,0)\,x^{-\b/2}\,,
\qquad \mbox{as } x\to\infty\,\,.
\ee
\item If $\gamma\in(0,1)$, then $\hat{\g}=\gamma$ and we have $\E[e^{-sY_\cT}]=\E[e^{-s^\g(c+o(1))S_\cT}]$  from which
\be\label{leapoverY-app-case2}
  \P(Y_\cT > x)  \sim \frac{\sqrt{\nu}c^{\b/2}}{\G(1-\g\b/2)} \,\F(1,0)\,x^{-\g\b/2}\,,
\qquad \mbox{as } x\to\infty\,\,.
\ee
\end{itemize}
\end{proof}

\subsubsection{ First-ladder length.}\label{GP}
\begin{proof}[Proof of Corollary \ref{lengthY}]
Recall that $L_{\cT}(Y)$ can be seen as the value at time $\cT$ of a
RWRSB driven by $S$ and with scenery $\z^+=\z^-=\z$.
Notice that with this choice  $ \hat\g\equiv \hat{\g}_+=\hat{\g}_-=\hat{\g}_0$
and therefore $\r_+\ge\r_0\,$. Following Theorem \ref{MainResultExtended},
we derive the next distinct cases.
\begin{itemize}
\item  If $\g\in(1,2]\,$, that is $\hat\g_0=\hat\g_+=1$, we obtain:
\medskip
\begin{enumerate}[label=(\roman*)]
\item If $E(|\xi_1|)<\infty$, which means $\b\in(1,2)$ and $\r_0<\r_+$, then
$L_{\cT}(Y)$ is dominatedly varying with index $1/2$,
$$
\P(L_{\cT}(Y) > x) \asymp x^{-1/2}\,,
\qquad \mbox{as } x\to\infty\,\,;
$$
\item If $\hat \b\equiv\b=1$, that is $\r_0=\r_+=1$, then
$$
\P(L_{\cT}(Y) > x) \asymp \sqrt{\log(x)}x^{-1/2}\,,
\qquad \mbox{as } x\to\infty\,\,.
$$
\end{enumerate}
\item If $\g\equiv\g_0=\g_+\in(0,1)$, instead, we have:
\medskip
\begin{enumerate}[label=(\roman*)]
\item If $E(|\xi_1|)<\infty$, that is $\b\in(1,2)$,
$$
k_{low}\cdot x^{-\min\left\{\frac 1 2,\frac{\g\b}2 \right\}}  \le \P(L_{\cT}(Y) > x)  \leq k_{up}\cdot x^{-\g/2}\,, \qquad \mbox{as } x\to\infty\,;
$$
\item If $\hat \b\equiv\b=1$, which means $\r_0=\r_+=\g$, then 
$$
 k_{low}\cdot x^{-\frac{\g}2 }  \le \P(L_{\cT}(Y) > x)  \leq k_{up}\cdot \sqrt{ \log(x)}\,x^{-\g/2}\,, \qquad \mbox{as } x\to\infty\,\,.
$$
\end{enumerate}
\end{itemize}
\end{proof}
\subsubsection{Continuous first-passage time  for  the generalized L\'evy-Lorentz gas}\label{sec: appLLG}
\begin{proof}[Proof of Corollary \ref{gLG}]
Recall that  the continuous first-passage time $\cT(X)=L_\cT(Y)-Y_\cT$
can be seen as the value at time $\cT$ of a  RWRSB driven by $S$ and
with scenery $\z^+ \equiv 0$ and $\z^-=2\z$. Notice that with this choice
 $\hat{\g}\equiv\hat{\g}_-=\hat{\g}_0 < \hat{\g}_+= +\infty$, and hence
 $\min\{\r_0, \r_+\}=\r_0=\hat\g\hat\b/2 $.
 As a consequence, Eq.~\eqref{splitting - gen} simply becomes
$$ \mathcal{G}_{L_\cT(Y)-Y_\cT}(s)= \mathbb{E} \left[ \prod_{k \leq 0} \mathcal{G}_{\z_1}(s \cN_{\cT}(k))\right] \,.$$
As before, following Theorem \ref{MainResultExtended},
we have to study different cases:
\begin{itemize}
\item  If $\g\in(1,2]\,$, that is $\hat{\g}=1$, we obtain the following results.\smallskip
\begin{enumerate}[label=(\roman*)]
\item If $E(|\xi_1|)<\infty$, with $\b\in(1,2)$ and $\hat\b=1$, then $\cT(X)$ is dominatedly varying of order $1/2$
$$
\P(\cT(X) > t) \asymp t^{-1/2}\,,
\qquad \mbox{as } t\to\infty\,\,;
$$
\item If $\hat\b\equiv\b=1$, then
$$
\P(\cT(X) > t) \asymp \sqrt{\log(t)}t^{-1/2}\,,
\qquad \mbox{as } t\to\infty\,\,.
$$
\end{enumerate}
\item If $\g=\hat\g\in(0,1)$,  we get: \smallskip
\begin{enumerate}[label=(\roman*)]
\item If $E(|\xi_1|)<\infty$, that is $\b\in(1,2)$, we can conclude that
$$
K_{low}(t) \cdot t^{-\min\left\{\g,\frac 1 2 \right\}}  \le \P(\cT(X) > t)  \leq k_{up}\cdot t^{-\g/2}\,, \qquad \mbox{as } t\to\infty\,\,,
$$
with $K_{low}(t)=K_{low}\log(t)$ if $\g=1/2\,$, constant otherwise.\\
\item If $\hat\b\equiv\b=1$, we obtain
$$
K_{low} (t)\cdot t^{-\min\left\{\g,\frac 1 2 \right\}}  \le \P(\cT(X) > t)  \leq k_{up}\cdot\sqrt{\log(t)} t^{-\g/2}\,, \qquad \mbox{as } t\to\infty\,\,,
$$
with
$$K_{low}(t)=k_{low}\begin{cases}
\sqrt{\log(t)}\qquad & \mbox{if}\quad\g>1/2\,,\\
 \log^{3/2}(t) \qquad & \mbox{if}\quad\g=1/2\,,\\
1\qquad & \mbox{if}\quad\g<1/2\,,
\end{cases}$$
\end{enumerate}
\end{itemize}
\end{proof}

\vspace{0.3cm}

\bigskip
{\bf Acknowledgements}
The work of A.~Bianchi is partially funded by the University of Padova through
the BIRD project 198239/19  \emph{``Stochastic processes and applications to disordered systems''}.
 The work of  G.~Cristadoro and  G.~Pozzoli is partially 
funded by the PRIN Grant 2017S35EHN \emph{``Regular and stochastic behavior in dynamical systems''} from the Ministero dell'Universit\`a e della ricerca (MIUR, Italy). The research of G.~Pozzoli is supported by the CY Initiative of Excellence through the grant Investissements d'Avenir ANR-16-IDEX-0008; part of this work was done while G.~Pozzoli was a PhD student at Universit\`a degli studi dell'Insubria (Dipartimento di Scienza e Alta Tecnologia and Center for Nonlinear and Complex Systems), and supported by Istituto Nazionale di Fisica Nucleare - Sezione di Milano.

\bigskip

\appendix

\section{Lemma}
\begin{lemma}\label{Wiener-Hopf} Let
\begin{align}
P(z, t)=\sum_{n=0}^{\infty} z^{n} p_{n}(t), \qquad & Q(z, t)=\sum_{n=0}^{\infty} z^{n} q_{n}(t) \,,\\
P^{*}(z, t)=\sum_{n=0}^{\infty} z^{n} p_{n}^{*}(t), \qquad& Q^{*}(z, t)=\sum_{n=0}^{\infty} z^{n} q_{n}^{*}(t)\,,
\end{align}
where $p_{0}(t) \equiv q_{0}(t) \equiv p_{0}^{*}(t) \equiv q_{0}^{*}(t) \equiv 1$; and for $n \geq 1$, $p_{n}$ and $p_{n}^{*}$ as functions of $t$ are Fourier transforms of measures with support in $(0, \infty)$; $q_{n}$ and $q_{n}^{*}$ as functions of $t$ are Fourier transforms of measures in $(-\infty, 0]. $ Suppose that for some $z_{0}>0$ the four power series converge for $z$ in $\left(0, z _{0}\right)$ and all real $t,$ and the identity
$$
P(z, t) Q^{*}(z, t) \equiv P^{*}(z, t) Q(z, t)
$$
holds there. Then
$$
P \equiv P^{*}, \quad Q \equiv Q^{*}\,.
$$
\end{lemma}
\proof{see Section 8.4.1 of \cite{Chung2001}}.
\qed

\section{Probabilistic tools}\label{app-prob}
\subsection{Basic definitions and results for random variables in the domain of attraction of a $\gamma$-stable distribution}\label{app-stable}

We will refer to~\cite[Ch.~2]{IL71}\cite{J, bingham} and references therein.
A random variable $X$ is $\gamma$-stable, with $\gamma\in (0,2]$, if $\forall\, s \in \R$ it has characteristic function  
\begin{multline}\label{eq:CharStable}
    \phi_X(s)
    \coloneqq\E[e^{isX}]
    =  \exp{\left\{-c |s|^\gamma\left[ 1-i\theta \,\sgn(s) \mathfrak{w}(s,\gamma)\right]+i\mu s\right\}}\\
    \text{with}\quad
    \mathfrak w(s,\gamma)\coloneqq
    \begin{cases}
        \tan(\tfrac \pi 2 \gamma),\quad & \gamma\neq 1,\\
       - \tfrac 2 \pi \log |s|,\quad & \gamma= 1,
    \end{cases}
\end{multline}
where $\mu\in\R$, $c> 0$, and $|\theta|\leq 1$ is a skewness parameter. Note that $\mu\equiv \E[X]$ when $\gamma>1$. In the limit case $\gamma=2$, $\theta$ is irrelevant and usually set equal to zero. If $\theta=1$, $X$ is \emph{spectrally positive} and the generating function is also well-defined and given by, $\forall \, s\geq 0,$ 
\begin{equation}\label{eq:GenStable}
     \cG_X(s)
    \coloneqq \E[e^{-sX}]
    = \begin{cases}
        \exp{\left\{ -\tfrac c{\cos \tfrac\pi 2 \gamma} s^\gamma-\mu s\right\}},\quad & \gamma \neq 1,\\
        \exp{\left\{\tfrac 2 \pi c s\log s -\mu s\right\}},\quad & \gamma = 1\,.
    \end{cases}
\end{equation}

We say that a random variable $Y$ is \emph{in the domain of attraction} of a $\gamma$-stable law, with $\gamma\in(0,2]$, if~\cite[Thm.~2.6.5]{IL71} 
\begin{equation*}
    \phi_Y(s)=1+i\mu s-c|s|^\gamma\ell(s)[1-i\theta\,\sgn(s) \mathfrak w(s,\gamma)]+o(|s|^\gamma\ell(s)\mathfrak w(s,\gamma)) \quad\text{as}\quad s \to 0,
\end{equation*}
where $\ell(s)$ is a positive \emph{slowly varying function}\footnote{$\ell(s)$ is a slowly varying function as $s\to 0$ if, for all $a>0$, $\lim_{s\to 0} \ell(as)/\ell(s)=1$.} at zero. In the case of a slowly varying function $\ell(s)$ that is merely a constant, the term \emph{normal} domain of attraction is used~\cite[Thm.~2.6.6-7]{IL71}.
Let us stress that $\theta=0$ if $Y$ has a symmetric distribution, and $\theta\in\{-1,+1\}$ if $Y$ has a one-sided distribution. In particular, for $\gamma\in(0,1)$ and $Y\geq 0$,
\begin{align*}
\phi_{\pm Y}(s)
&=1-c'e^{\mp i\tfrac \pi 2 \gamma} \ell(s)s^\gamma+o(\ell(s)s^\gamma),\quad c'=\tfrac{c}{\cos(\tfrac\pi 2 \gamma)},\quad \text{as}\quad s\to 0^+.
\end{align*}

As long as $\gamma\in(0,2)$, we can equivalently characterize the random variable $Y$ by saying that the tails of the distribution are \emph{regularly varying}\footnote{ $f(x)$ is \emph{regularly varying} of index $\rho$ as $x\to \infty$ if $f(x)=x^\rho \ell(x)$ for some slowly varying function $\ell(x)$ at infinity. Equivalently, $\lim_{x\to \infty} f(ax)/f(x)=a^\rho$ for all $a>0$.} of index $-\gamma$ at infinity~\cite[Thm.~2.6.1]{IL71}.

The relationship between the tails of the distribution function and the behavior around zero of the characteristic and generating functions is established by Abelian and Tauberian theorems for Fourier and Laplace-Stieltjes transforms, respectively. As we make extensive use of the Tauberian direction in the main text, we will provide ourselves with easy reference to explicit formulae.

Let $\gamma\neq 1 $ for simplicity. By the Karamata's Tauberian Theorem~\cite[Thm.~1.7.6]{bingham} for Laplace-Stieltjes transforms and~\cite[Thm.~1.7.2]{bingham}, we recall that by defining the possibly centred random variable 
\[
\tilde Y\coloneqq \begin{cases}
    Y\quad &\text{if}\quad \gamma\in(0,1)\\
    Y-\E[Y]\quad &\text{if}\quad \gamma\in(1,2)
\end{cases}\,,
\]
we have, for some \emph{positive} slowly varying function $\ell$,
\begin{multline*}
\cG_{\tilde Y}(s)=1-\Gamma(1-\gamma)\ell(s)s^\gamma+o(\ell(s)s^\gamma) \quad\text{as}\quad s\to 0^+\\
\implies\quad \P(Y>x)\sim \ell(1/x)x^{-\gamma}\quad\text{as}\quad x \to \infty.
\end{multline*}
As a first comment, observe that the change of sign of $\Gamma(1-\gamma)$ when $\gamma>1$ is consistent with~\eqref{eq:GenStable}. Secondly, a direct extension of the Tauberian theorem to \emph{dominated variation}\footnote{ A function is of dominated variation if it is O-regularly varying and monotone. $f(x)$ is \emph{O-regularly varying} if there exist constants $C>1$, $\gamma_1$, $\gamma_2$, $x_0$ such that
\[
\tfrac 1 C a^{\gamma_1}\leq \frac{f(ax)}{f(x)}\leq Ca^{\gamma_2},\quad a\geq 1, \quad x\geq x_0.
\]} 
can be found in~\cite[Thm.~2.10.2]{bingham}.

By the Tauberian theorems for Fourier kernels~\cite[\S~8.1.4]{bingham}, instead, if $\gamma\neq 1$ and for some positive slowly varying function $\ell$ and some $\tilde c\in\C$ with $\Re \tilde c>0$ (in fact $\tilde c\equiv [1-i\theta\,\tan(\tfrac\pi 2\gamma)]$), we can write\footnote{We use the fact that
$
\frac 2 \pi \sin(\tfrac\pi 2 \gamma)\Gamma(\gamma)=\frac 1 {\cos(\tfrac\pi 2\gamma)}\frac{\sin (\pi \gamma)\Gamma(\gamma)}\pi=\frac 1 {\cos(\tfrac\pi 2\gamma)\Gamma(1-\gamma)}.
$}
\begin{align}\label{ape}
&\phi_{\tilde Y}(s)=1-\tilde c\ell(s)s^\gamma+o(\ell(s)s^\gamma) \quad\text{as}\quad s\to 0^+\\
&\nonumber
\hspace{3cm} \implies\quad \P(Y>x)\sim p_+\frac{\ell(1/x)}{\cos(\tfrac\pi 2\gamma)\Gamma(1-\gamma)}x^{-\gamma}\quad\text{as}\quad x \to \infty,\\
&\nonumber
\hspace{3cm} \implies\: \P(Y<-x)\sim p_-\frac{\ell(1/x)}{\cos(\tfrac\pi 2\gamma)\Gamma(1-\gamma)}x^{-\gamma}\quad\text{as}\quad x \to \infty,
\end{align}
where $p_++p_-=1$ with
\[
p_+=\frac{1+\theta}2=\frac 1 2 -\frac{(\Im \tilde c)}{2(\Re \tilde c) \tan(\tfrac\pi 2 \gamma)}\in[0,1].
\]
If $p_\pm=0$, we interpret the result as $o(\ell(1/x)x^{-\gamma})$.
A similar result is obtained for the limit case $\gamma=1$ (see~\cite[Thm.~17(a)]{Soni3}, and~\cite[Thm.~18(a)]{Soni3} with an additional assumption on $\ell$ --- refer to~\cite[Thm.~3.6.8]{bingham}). In particular, when $\Im \tilde c=o(\ell(s)s^\gamma)$ (e.g.~in Proposition~\ref{prop-lengthRW}(D)), we simply have to replace $1/\cos(\gamma\pi/2)$ by $2/\pi$ in \eqref{ape}
(see \cite[Thm.~8.1.10]{bingham}), with the slight abuse of notation $\Gamma(0)\coloneqq 1$.

In light of the aforementioned theorems, as a final comment, let us stress the consistency between~\eqref{eq:CharStable} and~\eqref{eq:GenStable}.

\subsection{Estimates on joint characteristic/generating functions}
\begin{lemma}\label{lemma cauchy}
Assume that, for $k\in \{1,2\}$, $Z_k(s)$ is a complex random variable defined by $Z_k(s)\coloneqq e^{isX_k}$ or $Z_k(s)\coloneqq e^{-sX_k}$, whose average therefore corresponds to the characteristic or generating function of a real or 
non-negative random variable $X_k$.

If $|\E[Z_k(s)]|=1-c_k s^{\g_k}+o(s^{\g_k})$, with $\g_k\in (0,1)$, $c_k\in \mathbbm{R}^+$ and $s\to 0^+$, then by defining $\g\coloneqq \min\{\g_1,\g_2\}$ and assuming $c_1\neq c_2$ if  $\g_1=\g_2$ and $\Im(Z_1(s)),\Im(Z_2(s))\neq 0\,$, we get
\be
1-k_+s^\g+o(s^\g)
\leq|\E[Z_1Z_2(s)]|\leq1-k_-s^\g+o(s^\g)\,,\nonumber
\ee
where the positive constants $k_+\geq k_-$,  matching if $\g_1\neq\g_2$, are functions of $c_1,\,c_2$.
\end{lemma}

\begin{proof}
To easy the notation, from now on we will drop the dependence on $s$ of $Z_k$'s. By definition
$$\E[Z_1Z_2]=\E[Z_1]\E[Z_2]+\cov(Z_1,\bar{Z}_2)\,,$$
where $\bar{Z}_2$ denotes the complex conjugate of $Z_2$ and
$$\cov(Z_1,\bar{Z}_2)\coloneqq \E[(Z_1-\E[Z_1])(Z_2-\E[Z_2])]\,.$$
From the Cauchy-Schwarz inequality, we have
$$|\cov(Z_1,\bar{Z}_2)|\leq \sqrt{\var(Z_1)\var(Z_2)}\,,$$
where
$$\var(Z_k)\coloneqq \E[|Z-\E[Z]|^2]=\E[|Z_k|^2]-|\E[Z_k]|^2\,.$$
In particular, it holds that
\begin{align*}
|\E[Z_1Z_2]|&\leq |\E[Z_1]\E[Z_2]|+|\cov(Z_1,\bar{Z}_2)|\leq |\E[Z_1]\E[Z_2]|+\sqrt{\var(Z_1)\var(Z_2)}\,,\\
|\E[Z_1Z_2]|&\geq |\E[Z_1]\E[Z_2]|-|\cov(Z_1,\bar{Z}_2)|\geq |\E[Z_1]\E[Z_2]|-\sqrt{\var(Z_1)\var(Z_2)}\,.
\end{align*}
Since by assumptions
\be
\var(Z_k)= \E[|Z_k|^2]-[1-c_ks^{\g_k}+o(s^{\g_k})]^2\,,\nonumber
\ee
to determine the behavior of the variance, as $s\to 0^+$, we have to consider
two possible cases:
\begin{itemize}
\item If $Z_k(s)=e^{-sX_k}\,$, then
$$ \E[|Z_k|^2]=\E[Z_k^2]=\E[e^{-2s\,X_k}]=1-c_k(2s)^{\g_k}+o(s^{\g_k})\,,$$
and hence $\var(Z_k)=(2-2^{\g_k})c_ks^{\g_k}+o(s^{\g_k})\,$, with $2-2^{\g_k}\in(0,1)\,$;
\item  If $Z_k(s)=e^{isX_k}\,$, then $\E[|Z_k|^2]=1$, which implies $\var(Z_k)=2c_ks^{\g_k}+o(s^{\g_k})\,$.
\end{itemize}
Let us stress that even if $\E[Z_1],\,\E[ Z_2]$ are both characteristic functions, we get
$$k_-=c_1+c_2- 2\sqrt{c_1c_2}\geq 0\,,$$
with $k_->0$ whenever $c_1\neq c_2\,$.
The statement is therefore proved.
\end{proof}
Notice that Lemma \ref{lemma cauchy} can be easily generalized to a limiting case that is useful for the proof of Lemma~\ref{lemma 1} when $\b=2\,$.
Indeed, assuming that $\g_1< 1\leq \g_2$ and $\E[Z_2]$ is the characteristic function of a non-negative random variable $X_2$, then
$$|\E[Z_1Z_2]|\sim |\E[Z_1]|=1-c_1s^{\g_1}+o(s^{\g_1})\,$$
still holds true.

Moreover, if  we focus on the characteristic functions, we can extend the result to the range $\g_k\in(0,2]$ with $\g\coloneqq \min\{\g_1,\g_2\}<2$.

\subsection{Tail asymptotic of the product of independent random variables}
\begin{lemma}\label{lemma product}
Let $V,\,W$ be non-negative independent random variables characterized by the  asymptotic tails
$$ \P[V>v]\sim c_V\cdot [\log(v)]^{k_V}v^{-\g_V}\,,\qquad \P[W>w]\sim c_W \cdot w^{-\g_W}\,,\qquad \mbox{as}\quad v,w\to+\infty\,,$$
with $\g_V,\g_W\in(0,2)$, $k_V\geq 0$ and $c_V,c_W>0$. It holds that \cite{KV}
\begin{itemize}
\item[(A)] If $\g_V<\g_W$, then
$$\P[V\cdot W>z]\sim c_V\cdot \E[W^{\g_V}]\cdot [\log(z)]^{k_V} z^{-\g_V}\,,\qquad \mbox{as}\quad z \to +\infty\,;$$
\item[(B)] If $\g_V=\g_W\eqqcolon \g$, then
$$ \P[V\cdot W>z]\sim \frac{\g c_Vc_W}{k_V+1}\cdot [\log(z)]^{1+k_V} z^{-\g}\,,\qquad \mbox{as}\quad z \to +\infty\,.$$
\end{itemize}
\end{lemma}

\end{document}